\RequirePackage{fix-cm}
\documentclass[smallextended]{svjour3}       
\smartqed  
\usepackage{graphicx}
\usepackage{graphicx,epsfig}
\usepackage{enumerate}
\usepackage{amssymb}

\usepackage{xcolor}
\usepackage{microtype}

\usepackage{mathrsfs}
\usepackage[utf8]{inputenc}
\usepackage{bm}
\usepackage{mathtools}
\usepackage{amsmath}
\usepackage{amsbsy}
\usepackage{mathrsfs}
\usepackage{cite}
\usepackage{float} 
\usepackage{xcolor}
\usepackage{tikz}
\usetikzlibrary{positioning}
\usepackage{pgfplots}
\usepackage{pst-plot}
\usepackage{amsmath}
\usepackage{amsfonts}
\usepackage{amssymb}
\usepackage[top=3.1cm,bottom=3.2cm,left=3.5cm,right=3.5cm]{geometry}

\definecolor{mypink1}{rgb}{0.858, 0.188, 0.478}
\definecolor{mypink2}{RGB}{219, 48, 122}
\definecolor{mypink3}{cmyk}{0, 0.7808, 0.4429, 0.1412}
\definecolor{mygray}{gray}{0.6}
\definecolor{mylightred}{RGB}{255,200,200}
\definecolor{mylightblue}{RGB}{30,144,255}
\definecolor{mylightgreen}{RGB}{0,153,0}

\DeclareMathOperator{\dis}{dist}
\DeclareMathOperator{\inte}{int}
\DeclareMathOperator{\dom}{dom}
\DeclareMathOperator{\cl}{cl}

\smartqed 
\usepackage[linkcolor=blue, urlcolor=blue, citecolor=blue,colorlinks, bookmarks]{hyperref} 
\journalname{Soft Computing}
\begin{document}

\title{Weak sharp minima for interval-valued functions and its primal-dual characterizations using generalized Hukuhara subdifferentiability
}

\titlerunning{WSM for IVFs and its characterizations using $gH$-subdifferentiability}

\author{Krishan Kumar* \and Debdas Ghosh     \and Gourav Kumar  
}

\institute{Debdas Ghosh,  (debdas.mat@iitbhu.ac.in) \at Department of Mathematical Sciences, Indian Institute of Technology (Banaras Hindu University) Varanasi, Uttar Pradesh  221005, India
\and Krishan Kumar *Corresponding author (krishankumar.rs.mat19@itbhu.ac.in) \at Department of Mathematical Sciences, Indian Institute of Technology (Banaras Hindu University) Varanasi 
\and Gourav Kumar (gouravkr.rs.mat17@iitbhu.ac.in) \at  Department of Mathematical Sciences, Indian Institute of Technology (Banaras Hindu University) Varanasi 
}


\date{Received: date / Accepted: date}

\maketitle

\begin{abstract}
This article introduces the concept of weak sharp minima (WSM) for convex interval-valued functions (IVFs). To identify a set of WSM of a convex IVF, we provide its primal and dual characterizations. The primal characterization is given in terms of $gH$-directional derivatives. On the other hand, to derive dual characterizations, we propose the notions of the support function of a subset of $I(\mathbb{R})^{n}$ and $gH$-subdifferentiability for convex IVFs. Further, we develop the required $gH$-subdifferential calculus for convex IVFs. Thereafter, by using the proposed $gH$-subdifferential calculus, we provide dual characterizations for the set of WSM of convex IVFs.

\end{abstract}

\section{Introduction}
Due to the presence of uncertainty,  deterministic optimization fails to represent many real-life optimization problems. In such cases, we need to proceed with the tools of uncertain optimization. If the uncertainty is given by a random variable, then these optimization problems come under the umbrella of stochastic optimization. On the other, if the uncertainty is given by a membership function, then these optimization problems are solved with the techniques of fuzzy optimization. Also, it is seen that the uncertainty of many practical problems is expressed using closed and bounded intervals. Thus, interval optimization is an indispensable way to deal with the uncertainty present in many real-life problems.

In 1966, Moore \cite{moore1966interval} introduced interval analysis to investigate interval-valued functions (IVFs). In \cite{moore1966interval}, Moore extensively gave arithmetic of intervals. Subsequently, there was a need to improve this arithmetic \cite{hukuhara1967integration}, especially the subtraction. 
Due to this, Hukuhara \cite{hukuhara1967integration} presented a notion of the difference between intervals, which is known as Hukuhara difference ($H$-difference). Stefanini and Bede \cite{stefanini2009generalized} proposed an extended version of $H$-difference, known as generalized Hukuhara difference ($gH$-difference), which has been comprehensively adopted in interval analysis.

We know that the solution concepts of optimization problems depend widely on ordering the range set of the objective function. Unlike real numbers, the set of intervals is not linearly ordered. Thus, to introduce a solution concept for optimization problems under interval uncertainties, many partial ordering relations of intervals were proposed in the literature (see \cite{ishibuchi1990multiobjective,sengupta2001interpretation,wu2008interval,wu2008wolfe,jiang2008nonlinear,bhurjee2012efficient}, and the references therein). With the help of the existing ordering concepts of a pair of intervals, many theories and methods have been developed regarding solutions of optimization problems with IVFs or of interval optimization problems (IOPs) \cite{shaocheng1994interval,inuiguchi1995minimax,mraz1998calculating,wu2009karush,wu2010duality,jana2014solution}. Inuiguchi et al.  \cite{inuiguchi1995minimax} proposed treatment of optima for IVFs by minimax regret criteria. Chanas and Kutcha \cite{chanas1996multiobjective} gave a solution concept based on a preference relation of intervals. A robust efficient solution for interval linear porgramming was given in \cite{ida2003portfolio}. Chen et al. \cite{chen2004interval} reported a solution concept by midpoint deterministic approach. Wu \cite{wu2007karush} defined type I and type II $LU$-optimal solution concepts similar to Pareto optimality. A survey on the different ordering of intervals and related optimality concepts can be found in  \cite{jiang2012new,ghosh2020variable} and from their references.

The major developments on IOPs started after a rich calculus of IVFs was ready to be used. Hukuhara \cite{hukuhara1967integration} laid the foundation to develop the calculus of IVFs by introducing the concept of $H$-differentiability of IVFs. However, this definition of $H$-differentiability is restrictive \cite{chalco2013calculus}. To overcome the deficiencies of $H$-differentiability, Stefanini and Bede \cite{stefanini2009generalized} proposed the notion of $gH$-differentiability for IVFs. Later, by using $gH$-differentiability, many concepts on calculus have been developed, for instance, see \cite{chalco2013calculus,lupulescu2013hukuhara,ghosh2020generalized}. In 2007, Wu \cite{wu2007karush} proposed two solution concepts by considering two partial ordering concepts on the set of all closed intervals and derived KKT optimality conditions for IOPs using $H$-derivative. Subsequently, Wu investigated KKT optimality conditions for multi-objective IOPs \cite{wu2009karush}. In 2012, Bhurjee et al. \cite{bhurjee2012efficient} developed a methodology to study the existence of the solution of general IOPs by expressing IVFs in the parametric form. Chalco-cano et al. \cite{chalco2013optimality} derived KKT optimality conditions for IOPs using $gH$-derivative and explained the advantages of using $gH$-derivative instead of $H$-derivative.  In 2016, Singh et al.  \cite{singh2016kkt}  proposed the concept of Pareto optimal solution for the interval-valued multi-objective programming problems. Many other researchers have also proposed optimality conditions and solution concepts for IOPs, see for instance \cite{ahmad2019sufficiency,ghosh2017newton,ghosh2018saddle,ghosh2019extended,treanctua2021class} and the  references therein.

\subsection{Motivation and work done}
To analyse optimization problems with multiple minima points, Burke and Ferris \cite{burke1993weak} introduced the notion of WSM in conventional optimization. It is known that WSM plays an important role in the sensitivity analysis and convergence analysis of conventional optimization problems \cite{burke1993weak,zhou2012new,ferris1990iterative}.  It is also seen that many algorithms exhibit finite termination at WSM \cite{burke1993weak,ferris1990iterative,zhou2012new,matsushita2012finite,wang2015finite}. Motivated by these properties and wide applications of WSM in conventional optimization, in this article, we attempt to propose and mathematically characterize the notion of WSM for convex IVFs. To give characterizations of WSM for convex IVFs, we defined $gH$-subdifferentiability for convex IVFs and support function of a subset of $I(\mathbb{R})^{n}$. Some required fundamental characteristics of $gH$-subdifferential set are proposed. Few related results on the support function of a nonempty subset of $I(\mathbb{R})^{n}$ are also derived.

\subsection{Delineation}
The article is presented in the following manner. In Section \ref{Sec2}, basic terminologies and definitions on intervals and IVFs are provided. In Section \ref{Sec3}, we propose the concept of the support function of a subset of $I(\mathbb{R})^{n}$; alongside, few necessary results on extended support function are also given. Next, we derive the idea of $gH$-subdifferentiability for  convex IVFs, in Section \ref{sec4_new} that are required in the subsequent sections.  The concept of  WSM for convex IVFs is presented in Section \ref{Sec4}; further, we give primal and dual characterizations of WSM. Lastly, the conclusion and future scopes are given in Section \ref{Sec5}.  

\section{Preliminaries and Terminologies}\label{Sec2}
In this article, the following notations are used throughout. 
\renewcommand\labelitemi{\raisebox{0.5ex}{\tiny$\bullet$}}
\begin{itemize}
    \item $\mathbb{R}$ denotes the set of real numbers
    \item $\mathbb{R}^+$ denotes the set of nonnegative real numbers
     \item $\lVert \cdot \rVert$ denotes the Euclidean norm  and $\langle \cdot,\cdot \rangle$ denotes the standard inner product on $\mathbb{R}^{n}$
    \item $I(\mathbb{R})$ represents the set of all closed and bounded intervals
    \item  Bold capital letters refer to the elements of $I(\mathbb{R})$
    \item $\overline{I({\mathbb{R}})}=I(\mathbb{R})\cup\{-\infty,+\infty\}$
    \item $\mathbb{B}=\{x\in{\mathbb{R}}^n:\lVert x \rVert\leq 1\}$ denotes the closed unit ball in ${\mathbb{R}}^n$.
    \end{itemize}

\subsection{Fundamental Operations on Intervals}

Arithmetic operations of two intervals ${\textbf A} = [\underline{a}, \overline{a}]$ and $\textbf{B} = \left[\underline{b}, \overline{b}\right]$ are defined by
$${\textbf A} \oplus {\textbf B} =\left[~\underline{a} + \underline{b},~ \overline{a} + \overline{b}~\right],~{\textbf A} \ominus {\textbf B} =  \left[~\underline{a} - \overline{b},~ \overline{a} - \underline{b}~\right],$$
and
	$$\lambda \odot  {\textbf A} =  {\textbf A} \odot  \lambda = 
	\begin{cases}
	[\lambda \underline{a},~\lambda \overline{a}], & \text{if $\lambda \geq 0$}\\
	[\lambda \overline{a},~\lambda \underline{a}], & \text{if $\lambda < 0,$}
	\end{cases}$$
where $\lambda$ is a real constant.\\
The norm \cite{moore1966interval} of an interval ${\textbf A} = [\underline{a}, \overline{a}]$ in $I(\mathbb{R})$ is defined by 
${\lVert\textbf{A}\rVert}_{I(\mathbb{R})}=\max\{\lvert\underline{a}\rvert, \lvert\overline{a}\rvert\}.$\\
The norm of an interval vector $\widehat{\textbf{A}}=(\textbf{A}_1,\textbf{A}_2,\ldots,\textbf{A}_n)\in I(\mathbb{R})^n$ is given by (see \cite{moore1966interval})
\[\lVert \widehat{\textbf{A}}\rVert_{I(\mathbb{R})^n}=\sum_{i=1}^{n}\lVert{\textbf{A}_{i}}\rVert_{I(\mathbb{R})}.\]
 It is to note that a real number $p$ can be represented by the interval $[p,p].$

	\begin{definition}
	(\emph{gH-difference of intervals} \cite{stefanini2009generalized}). Let $\textbf{A}$ and $\textbf{B}$ be two elements in $I(\mathbb{R})$. The $gH$-difference
		between $\textbf{A}$ and $\textbf{B}$, denoted by $\textbf{A} \ominus_{gH} \textbf{B}$, is defined by the interval $\textbf{C}$ such that
		\[
		\textbf{A} =  \textbf{B} \oplus  \textbf{C} ~\text{ or }~ \textbf{B} = \textbf{A}
		\ominus \textbf{C}.
		\]
		It is to be noted that for $\textbf{A} = \left[\underline{a},~\overline{a}\right]$ and $\textbf{B} = \left[\underline{b},~\overline{b}\right]$,
		\[
		\textbf{A} \ominus_{gH} \textbf{B} = \left[\min\{\underline{a}-\underline{b},
		\overline{a} - \overline{b}\},~ \max\{\underline{a}-\underline{b}, \overline{a} -
		\overline{b}\}\right],
		\]
		and $\textbf{A} \ominus_{gH} \textbf{A} = \textbf{0}.$

	\end{definition}
	\begin{definition}\label{dd8}
	(\emph{Algebraic operations on $I(\mathbb{R})^{n}$}). Let $\widehat{\textbf{A}}=(\textbf{A}_{1},\textbf{A}_{2},\ldots,\textbf{A}_{n})$ and $\widehat{\textbf{B}}=(\textbf{B}_{1},\textbf{B}_{2}, \ldots,\textbf{B}_{n})$ be two elements in $I(\mathbb{R})^{n}$. An algebraic operation $\star$ between $\widehat{\textbf{A}}$ and $\widehat{\textbf{B}}$, denoted by $\widehat{\textbf{A}}\star\widehat{\textbf{B}}$, is defined by 
	\[
     \widehat{\textbf{A}}\star\widehat{\textbf{B}}=(\textbf{A}_{1}\star\textbf{B}_{1},\textbf{A}_{2}\star\textbf{B}_{2},\ldots,\textbf{A}_{n}\star\textbf{B}_{n}),
    \]
where $\star\in\{\oplus,~\ominus,~\ominus_{gH}\}$.
	\end{definition}

	\begin{definition}(\emph{Special product}).  
	For any $x=(x_{1},x_{2},\ldots,x_{n})\in\mathbb{R}^{n}$ and a vector of intervals  $\widehat{\textbf{A}}=(\textbf{A}_{1},\textbf{A}_{2},\ldots,\textbf{A}_{n}) \in I(\mathbb{R})^{n}$ with $\textbf{A}_{i}=[\underline{a}_{i},\overline{a}_{i}]$ for each $i=1,2,\ldots, n$, the special product between $x$ and $\widehat{\textbf{A}}$, denoted by $	{x}^{\top}\odot\widehat{\textbf{A}}$, is given by
	\[
	{x}^{\top}\odot\widehat{\textbf{A}}=\left[\min\left\{\sum\limits_{i=1}^{n} x_{i}\underline{a}_{i}, \sum\limits_{i=1}^{n} x_{i} \overline{a}_{i}\right\}, \max\left\{\sum\limits_{i=1}^{n} x_{i}\underline{a}_{i}, \sum\limits_{i=1}^{n} x_{i} \overline{a}_{i}\right\}\right].
	\]
\end{definition}

\begin{remark}\label{n1}
It is to notice that if all the components of $\widehat{\textbf{A}}$ are  degenerate intervals, i.e., $\widehat{\textbf{A}}\in\mathbb{R}^{n}$, then the special product  ${x}^{\top}\odot\widehat{\textbf{A}}$ reduces to the standard inner product of ${x}\in\mathbb{R}^{n}$ with $\widehat{\textbf{A}}$.
\end{remark}

	\begin{definition}\label{g99}(\emph{Dominance of intervals} \cite{wu2008wolfe}). Let $\textbf{A}=[\underline{a},\overline{a}]$ and $\textbf{B}=[\underline{b},\overline{b}]$ be two elements in $I(\mathbb{R})$. 
		\begin{enumerate}[(i)]
			\item $\textbf{B}$ is said to be dominated by $\textbf{A}$ if $\underline{a}\leq\underline{b}$ and $\overline{a}\leq\overline{b}$, and then we write $\textbf{A}\preceq \textbf{B}$;
			\item $\textbf{B}$ is said to be strictly dominated by $\textbf{A}$ if $\textbf{A}\preceq \textbf{B}$ and $\textbf{A}\neq\textbf{B}$, and then we write $\textbf{A}\prec\textbf{B}.$ Equivalently, $\textbf{A}\prec \textbf{B}$ if and only if any of the following holds:\\
			 `$\underline{a}<\underline{b}$ and $\overline{a}\leq \overline{b}$' \quad or \quad  
			 `$\underline{a}\leq \underline{b}$ and $\overline{a}<\overline{b}$' \quad or \quad  
			 `$\underline{a}<\underline{b}$ and $\overline{a}< \overline{b}$';
		
			\item if neither $\textbf{A}\preceq \textbf{B}$ nor $\textbf{B}\preceq \textbf{A}$, we say that none of $\textbf{A}$ and $\textbf{B}$ dominates the other, or $\textbf{A}$ and $\textbf{B}$ are not comparable. Equivalently, $\textbf{A}$ and $\textbf{B}$ are not comparable if either `$\underline{a}<\underline{b}$ and $\overline{a}>\overline{b}$' or `$\underline{a}>\underline{b}$ and $\overline{a}<\overline{b}$'.
		\end{enumerate}
	\end{definition}
	
\subsection{Calculus of IVFs}
Throughout this subsection, $\textbf{F}$ is an IVF defined on a nonempty subset $X$ of ${\mathbb{R}}^n.$ 
	\begin{definition}(\emph{${gH}$-continuity} \cite{ghosh2017newton}).
		Let $\textbf{F}$ be an IVF and let $\bar{x}$ be a point of ${X}$ and $h\in\mathbb{R}^{n}$ such that $\bar x+h\in X$. The function
		$\textbf{F}$ is said to be $gH$-continuous at $\bar{x}$ if
		\[
		\lim_{\lVert h \rVert\rightarrow 0}\left(\textbf{F}(\bar{x}+h)\ominus_{gH}\textbf{F}(\bar{x})\right)=\textbf{0}.
		\]
	\end{definition}

\begin{definition}(\emph{$gH$-derivative} \cite{stefanini2009generalized}). The $gH$-derivative of an IVF $\textbf{F}: \mathbb{R}\rightarrow I(\mathbb{R})$ at $\bar{x}\in \mathbb{R}$ is defined by
	\[
	\textbf{{F}}\unboldmath{'}(\bar{x})=\lim\limits_{h\to0}\frac{\textbf{F}(\bar{x}+h)\ominus_{gH}\textbf{F}(\bar{x})}{d},  \text{ provided the limit exists}.
	\]
	\end{definition}
	\begin{remark}\label{rm2}
	({See} \cite{stefanini2009generalized}). Let $\textbf{F}=[\underline{F},\overline{F}]$ be an IVF on $X$, where $\underline{F}$ and $\overline{F}$ are real-valued functions defined on $X$. Then, the $gH$-derivative of $\textbf{F}$ at $\bar{x}\in X$ exists if the derivatives of $\underline{F}$ and $\overline{F}$ at $\bar{x}$ exist and 
	\[
	\textbf{F}'(\bar{x})=\left[\min\left\{\underline{F}'(\bar{x}),\overline{F}'(\bar{x})\right\},
\max\left\{\underline{F}'(\bar{x}),\overline{F}'(\bar{x})\right\}\right]	
.\]
\end{remark}

	\begin{definition}
	(\emph{$gH$-partial derivative} \cite{ghosh2017newton}). Let $\bar{x}=(\bar{x}_{1},\bar{x}_{2},\ldots,\bar{x}_{n})^\top$ be  a point of $X$. For a given $i \in \{1,2,\ldots, n\}$, we define a function $\textbf{G}_{i}$ by  $\textbf{G}_{i}(x_{i}) =  \textbf{F}(\bar{x}_{1},\bar{x}_{2},\ldots,\bar{x}_{i-1}, x_{i},\bar{x}_{i+1},\ldots,\\\bar{x}_{n}).$  If $gH$-derivative of $\textbf{G}_{i}$ exists at $\bar{x}_{i}$, then we say that $\textbf{F}$ has the $i$th $gH$-partial derivative at $\bar{x}$.
     We denote the $i$th $gH$-partial derivative of $\textbf{F}$ at $\bar{x}$ by $D_{i}\textbf{F}(\bar{x})$, i.e.,   
     $
     D_{i}\textbf{F}(\bar{x})=\textbf{G}'_{i}(\bar x_{i})^\top.
     $
     \end{definition}
	\begin{definition}
	(\emph{gH-gradient} \cite{ghosh2017newton}). The $gH$-gradient of  $\textbf{F}$ at a point $\bar{x}\in X$, denoted by $\nabla\textbf{F}(\bar{x})\in I(\mathbb{R})^{n}$, is defined by 
	\[
	\nabla\textbf{F}(\bar{x})=(D_{1}\textbf{F}_{1}(\bar{x}),D_{2}\textbf{F}_{2}(\bar{x}),\ldots,D_{n}\textbf{F}_{n}(\bar{x}))^{\top}.
	\]
	\end{definition}
	\begin{lemma}\label{lm4}
	Let $\textbf{A}, \textbf{B}, \text{and}~ \textbf{C}$ are in $I(\mathbb{R})$. Then, for any real number $r$,
	\begin{enumerate}[\rm(i)]
	    \item \label{lm41} 
	    $[r,r]\preceq\textbf{A}$ and $\textbf{A}\preceq\textbf{B}\ominus_{gH}\textbf{C}$ $\implies$ $\textbf{C}\oplus[r,r]\preceq\textbf{B}$ and 
	   
	    \item \label{lm43}  $\left((1-\lambda)\odot\textbf{A}\oplus\lambda\odot \textbf{B}\right)\ominus_{gH}\textbf{A}=\lambda\odot\left(\textbf{B}\ominus_{gH}\textbf{A}\right)$ for any $\lambda\in[0,1].$
	\end{enumerate}
	\end{lemma}
	\begin{proof}
	        See Appendix \ref{aa1}.
	\end{proof}

	\begin{definition}(\emph{Convex IVF} \cite{wu2007karush}). Let $X$ be a nonempty convex subset of $\mathbb{R}^n$. An IVF 
		$\textbf{F}: {X} \rightarrow I(\mathbb{R})$ is said to be convex on ${X}$ if for any $x_1$ and $x_2$ in ${X}$,
		\[
		\textbf{F}(\lambda_1 x_1+\lambda_2 x_2)\preceq
		\lambda_1\odot\textbf{F}(x_1)\oplus\lambda_2\odot\textbf{F}(x_2)~\text{for all}~ \lambda_1,\lambda_2\in[0,1]~\text{with} ~\lambda_1+\lambda_2=1.
		\]
	
			\end{definition}
\begin{lemma}\label{lm2}
\emph{(See \cite{wu2007karush})}. Let $X$ be a nonempty convex subset of $\mathbb{R}^{n}$, and  $\textbf{F}=[\underline{F},\overline{F}]$ be an IVF on $X$, where $\underline{F}$ and $\overline{F}$ are real-valued functions defined on $X$. Then, $\textbf{F}$ is convex on $X$ if and only if $\underline{F}$ and $\overline{F}$ are convex on $X$.
\end{lemma}

		\begin{definition} \label{d1}
		(\emph{$gH$-directional derivative} \cite{ghosh2020generalized}). Let $\textbf{F}$ be an IVF on $X$. Let $\bar x\in {X}$ and $d\in\mathbb{R}^n$. If the limit
	\[\lim_{\lambda\to 0+}\frac{1}{\lambda}\odot\left(\textbf{F}(\bar{x}+\lambda d)\ominus_{gH}\textbf{F}(\bar{x})\right)
	\]
exists, then the limit is said to be $gH$-directional derivative of $\textbf{F}$ at $\bar{x}$ in the direction $d$, and it is denoted by $\textbf{F}_\mathscr{D}(\bar{x})(d)$.
		\end{definition}

        \begin{definition}
		(\emph{gH-differentiability}\label{dd11} \cite{ghosh2017newton}). An IVF $\textbf{F}$ is said to be $gH$-differentiable at $\bar{x}\in X$ if there exist two  IVFs $\textbf{E}(\textbf{F}(\bar{x});h)$ and $\textbf{L}_{\bar{x}}:\mathbb{R}^{n}\rightarrow I(\mathbb{R})$ such that 
		\[
		\textbf{F}(\bar{x}+h)\ominus_{gH} \textbf{F}(\bar{x})=\textbf{L}_{\bar{x}}(h)\oplus\lVert h\rVert\odot\textbf{E}(\textbf{F}(\bar{x});h)
		\]
for $\lVert h \rVert<\delta$ for some $\delta>0$, where $\lim\limits_{\lVert h\rVert\to0}$ $\textbf{E}(\textbf{F}(\bar{x});h)=\textbf{0}$ and $\textbf{L}_{\bar{x}}$ is such a function that satisfies 
		\begin{enumerate}[(i)]
		    \item $\textbf{L}_{\bar{x}}(x+y)=\textbf{L}_{\bar{x}}(x)\oplus\textbf{L}_{\bar{x}}(y)$ for all $x, y\in X$, and 
		    \item $\textbf{L}_{\bar{x}}(cx)=c\odot\textbf{L}_{\bar{x}}(x)$ for all $c\in\mathbb{R}$ and $x\in X.$
		\end{enumerate}
		\end{definition}
		\begin{theorem}\label{th3}\emph{(See \cite{ghosh2017newton})}.
		  Let $\textbf{F}: X\rightarrow I (\mathbb{R})$ be $gH$-differential at $\bar{x}$. Then, $\textbf{L}_{\bar{x}}$ exists for every $h\in\mathbb{R}^{n}$ and 
		  \[
		  \textbf{L}_{\bar{x}}=\sum\limits_{i=1}^{n} h_{i}\odot D_{i}\textbf{F}(\bar{x}),
		  \]
		  where $\sum\limits_{i=1}^{n} h_{i}\odot D_{i}\textbf{F}(\bar{x})=h_{1}\odot D_{1}\textbf{F}(\bar{x})\oplus h_{2}\odot D_{2}\textbf{F}(\bar{x})\oplus\cdots\oplus       h_{n}\odot D_{n}\textbf{F}(\bar{x}).$
		 \end{theorem}
		\begin{remark}\label{n3}{(See \cite{ghosh2017newton})}.    Let $\textbf{F}:X\rightarrow I(\mathbb{R})$ be  $gH$-differentiable at $\bar{x}\in X$. Then, there exists a nonzero $\lambda$ and $\delta>0$ such that 
		\[\lim_{\lambda\to 0}\frac{1}{\lambda}\odot\left(\textbf{F}(\bar{x}+\lambda h)\ominus_{gH}\textbf{F}(\bar{x})\right)=\textbf{L}_{\bar{x}}(h)
	    \]
		for all $h\in\mathbb{R}^{n}$ with $\rvert\lambda\lvert\lVert h \rVert<\delta$, where $\textbf{L}_{\bar{x}}$ is an IVF, defined in Definition \ref{dd11} of $gH$-differentiability.
		\end{remark}
		\begin{lemma}\label{lm8}
		Let $\textbf{F} : X\rightarrow I(\mathbb R)$ be a $gH$-differentiable at $\bar{x}\in X$. Then, $\textbf{F}$ has $gH$-directional derivative at $\bar{x}$ for every direction $h\in\mathbb{R}^{n}$ and 
		\[
		\textbf{F}_{\mathscr{D}}(\bar{x})(h)=\textbf{L}_{\bar{x}}(h) \text{ for all } h\in\mathbb{R}^{n},
		\]
		where $\textbf{L}_{\bar{x}}$ is as defined in Definition \ref{dd11}.
		\end{lemma}
		\begin{proof}
		Since \textbf{F} is $gH$-differentiable at $\bar{x}$, by Remark \ref{n3}, we have
		\begin{eqnarray*}
		&&\lim_{\lambda\to 0}\frac{1}{\lambda}\odot\left(\textbf{F}(\bar{x}+\lambda h)\ominus_{gH}\textbf{F}(\bar{x})\right)=\textbf{L}_{\bar{x}}(h) \text{ for all } h\in\mathbb{R}^{n}\\
		&\implies&\lim_{\lambda\to 0+}\frac{1}{\lambda}\odot\left(\textbf{F}(\bar{x}+\lambda h)\ominus_{gH}\textbf{F}(\bar{x})\right)=\textbf{L}_{\bar{x}}(h) \text{ for all } h\in\mathbb{R}^{n}.
		\end{eqnarray*}
		Hence, by Definition \ref{d1}, we conclude  that \textbf{F} has $gH$-directional derivative at $\bar{x}$ and 
		\[
		\textbf{F}_{\mathscr{D}}(\bar{x})(h)=\textbf{L}_{\bar{x}}(h) \text{ for all  } h\in\mathbb{R}^{n}.
		\]
\end{proof}

		\begin{definition}\label{dd7}(\emph{Proper IVF}). An extended IVF $\textbf{F}: X\rightarrow \overline{I(\mathbb{R})}$ is called a proper IVF if there exists $\bar x\in{X}$ such that $\textbf{F}(\bar x)\prec [+\infty,+\infty]$ and $[-\infty,-\infty]\prec\textbf{F}(x)~\text{for all}~ x\in{X}.$ 
		\end{definition}
	\begin{definition}
	(\emph{Effective domain of IVF}). The effective domain of an extended IVF $\textbf{F}: X\rightarrow \overline{I(\mathbb{R})}$ is the collection of all such points at which \textbf{F} is finite. It is denoted by $\dom(\textbf{F})$, i.e., 
	\[
	 \dom(\textbf{F})=\Big\{x\in X :\textbf{F}(x)\prec[+\infty,+\infty]\Big\}.
	\]
	\end{definition}

	\begin{theorem}\label{thm4}
		Let $X$ be a nonempty convex subset of $\mathbb{R}^{n}$ and $\textbf{F}: {X} \rightarrow I(\mathbb{R})$  be a convex IVF with $\textbf{F}(x)=[\underline{F}(x),\overline{F}(x)]$, where $\underline{F}$ and $\overline{F}$ are real-valued functions defined on $X$. Then, at any $\bar{x} \in X$, $gH$-directional derivative $\textbf{F}_\mathscr{D}(\bar{x})(d)$ exists  and 
		\[
		\textbf{F}_\mathscr{D}(\bar{x})(d)=\Big[\min\Big\{\underline{F}_{\mathscr{D}}(\bar{x})(d),\overline{F}_{\mathscr{D}}(\bar{x})(d)\Big\}, \max\Big\{\underline{F}_{\mathscr{D}}(\bar{x})(d),\overline{F}_{\mathscr{D}}(\bar{x})(d)\Big\}\Big].
		\]
	
	\end{theorem}
	\begin{proof}
	Similar to the proof of Theorem 3.1 in \cite{ghosh2020generalized}.
	\end{proof}
		\begin{definition}\label{dd12}
	(\emph{$gH$-Lipschitz continuous IVF} \cite{ghosh2020generalized}). An IVF $\textbf{F}$ is said to be $gH$-Lipschitz continuous on $X$ if there exists $M>0$ such that
	\[
	\lVert \textbf{F}(x) \ominus_{gH} \textbf{F}(y)\rVert_{I(\mathbb{R})} \leq M \lVert x-y \rVert \text{ for all } x,y\in X.
	\]
	The constant $M$ is called a Lipschitz constant.
	\end{definition}

\begin{definition}\label{12}(\emph{Supremum of a subset of $\overline{I(\mathbb{R})}$} \cite{gourav}). Let $\textbf{S}$ be a nonempty subset of ${{\overline{I(\mathbb{R})}}}$. An interval $\mathbf{\bar{A}}\in I(\mathbb{R})$ is said to be an upper bound of $\textbf{S}$ if $\textbf{B}\preceq \mathbf{\bar{A}}$ for all $\textbf{B}$ in $\textbf{S}$. An upper bound $\bar{\mathbf{A}} $ of $\textbf{S}$ is called a supremum of $\textbf{S}$, denoted by $\sup\textbf{S}$, if for all upper bounds $\textbf{C}$ of $\textbf{S}$ in $I(\mathbb{R}),~ \bar{\mathbf{A}}\preceq \textbf{C}$. Moreover, if the supremum of the set $\textbf{S}$ belongs to the set itself, then it is called maximum of $\textbf{S}$, denoted by $\max\textbf{S}$.
	\end{definition}
	\begin{remark}\label{nt5}{(See
	\cite{gourav})}. Let $\Lambda$ be an index set, and $\lambda\in\Lambda$.
	 For any subset $\textbf{S}=[a_\lambda,b_\lambda]$ of ${\overline{I(\mathbb{R})}}$,  we have  $\sup\textbf{S}=\left[\sup\limits_{\lambda\in\Lambda}a_\lambda,~\sup\limits_{\lambda\in\Lambda}b_\lambda\right]$.
	\end{remark}
	\begin{lemma}\label{g79}\emph{(See \cite{gourav})}.
		Let $\mathbf{F}_1$ and $\mathbf{F}_2$ be two proper extended IVFs defined on  $\mathcal{S}$, which is a nonempty subset of ${X}$. Then,
		\begin{enumerate}[\normalfont(i)]
			\item\label{g78} $\inf\limits_{x\in\mathcal{S}}\mathbf{F}_1(x)\oplus\inf\limits_{x\in\mathcal{S}}\mathbf{F}_2(x)\preceq\inf\limits_{x\in\mathcal{S}}\left(\mathbf{F}_1(x)\oplus\mathbf{F}_2(x)\right)$ and 
			\item\label{g77} $\sup\limits_{x\in\mathcal{S}}\left(\mathbf{F}_1(x)\oplus\mathbf{F}_2(x)\right)\preceq\sup\limits_{x\in\mathcal{S}}\mathbf{F}_1(x)\oplus\sup\limits_{x\in\mathcal{S}}\mathbf{F}_2(x).$
		\end{enumerate}
	\end{lemma}

			\begin{definition}\label{g5}(\emph{Lower limit and $gH$-lower semicontinuity of an extended IVF} \cite{gourav}).
		The lower limit of an extended IVF $\textbf{F}$ at $\bar{x}\in{X}$, denoted by $\liminf\limits_{x\rightarrow \bar{x}}\textbf{F}(x)$, is defined  by
		\begin{eqnarray*}\label{g6}
			\liminf_{x\rightarrow \bar{x}}\textbf{F}(x)&=&\lim\limits_{\delta\downarrow 0}(\inf\{\textbf{F}(x):x\in B_{\delta}(\bar{x})\})=\sup_{\delta>0}(\inf\{\textbf{F}(x):x\in B_{\delta}(\bar{x})\}),
		\end{eqnarray*}
		where $B_{\delta}(\bar{x})$ is an open ball with radius $\delta$ centered at $\bar{x}$.  
		$\textbf{F}$ is called $gH$-lower semicontinuous ($gH$-lsc) at a point $\bar{x}$ if 
		$\textbf{F}(\bar{x})\preceq\liminf\limits_{x\rightarrow \bar{x}}\textbf{F}(x)$. Further, \textbf{F} is called $gH$-lsc on ${X}$ if \textbf{F} is $gH$-lsc at every $\bar{x}\in {X}$.
	\end{definition}
	\begin{remark}\label{rm1}
	  By Note 5 of \cite{gourav}, we see that \textbf{F} is $gH$-lsc at $\bar{x}\in X$ if and only if $\underline{F}$ and $\overline{F}$ both are lsc at $\bar{x}$.   
	\end{remark}

		\begin{lemma}\label{lm3}
		Let $\textbf{F}$ : $\mathbb{R}^{n}\rightarrow {\overline{I(\mathbb{R})}}$ be a proper convex IVF. Then, for all ${x},y\in \dom(\textbf{F}),$ we have
		$$\textbf{F}_\mathscr{D}({x})(y-{x})\preceq \textbf{F}(y)\ominus_{gH}\textbf{F}({x}).$$
			\end{lemma}
		\begin{proof} By Definition \ref{d1} of $gH$-directional derivative, we have
			\begin{equation}\label{eq22}
			\textbf{F}_\mathscr{D}({x})(d)=
			\lim_{\lambda\to 0+}\frac{1}{\lambda}\odot\left(\textbf{F}({x}+\lambda d)\ominus_{gH}\textbf{F}({x})\right).
	\end{equation}
		By taking $d=y-x$ in (\ref{eq22}), we get
	\begin{eqnarray*}
		\textbf{F}_\mathscr{D}({x})(d)&=&\lim_{\lambda\to 0+}\frac{1}{\lambda}\odot\left(\textbf{F}({x}+\lambda(y-{x}))\ominus_{gH}\textbf{F}({x})\right)\\&=&\lim_{\lambda\to 0+}\frac{1}{\lambda}\odot\left(\textbf{F}((1-\lambda){x}+\lambda(y))\ominus_{gH}\textbf{F}({x})\right)\\&\preceq&\lim_{\lambda\to 0+}\frac{1}{\lambda}\odot\left\{((1-\lambda)\odot\textbf{F}({x})\oplus\lambda\odot\textbf{F}(y))\ominus_{gH}\textbf{F}({x})\right\}\\&=&\lim_{\lambda\to 0+}\frac{1}{\lambda}\odot\lambda\left(\textbf{F}(y)\ominus_{gH}\textbf{F}({x})\right)\text{ by (\ref{lm43})} \text{ of Lemma }\ref{lm4}\\&=&
		\textbf{F}(y)\ominus_{gH}\textbf{F}({x}).
		\end{eqnarray*}
		\end{proof}

	\begin{definition}\label{dd9}
	(Convergence of a sequence in $I(\mathbb{R})^{n}$). A sequence $\widehat{\textbf{G}}: \mathbb{N}\rightarrow I(\mathbb{R})^{n}$  is said to be convergent if there exists a $\widehat{\textbf{G}}\in I(\mathbb{R})^{n}$  such that 
	\[
	\lVert \widehat{\textbf{G}}_{k}\ominus_{gH} \widehat{\textbf{G}}\rVert_{I(\mathbb{R})^{n}}\rightarrow 0 \text{ as } k\rightarrow\infty,
	\]
	where $\widehat{\textbf{G}}({k})=\widehat{\textbf{G}}_{k}$, $k\in\mathbb{N}$.
	\end{definition}
\begin{remark}\label{nt1}
	It is noteworthy that if a sequence $\left\{\widehat{\textbf{G}}_{k}\right\} \text{ in } I(\mathbb{R})^{n},  \text{ where } {\widehat{\textbf{G}}_{k}}=(\textbf{G}_{k1},\textbf{G}_{k2},\ldots,\\\textbf{G}_{kn})\in I(\mathbb{R})^{n}$ with $\textbf{G}_{ki}=[\underline{g}_{ki}, \overline{g}_{ki}]$, converges to $\widehat{\textbf{G}}=(\textbf{G}_{1},\textbf{G}_{2},\ldots,\textbf{G}_{n})\in I(\mathbb{R})^{n}$ with $\textbf{G}_{i}=[\underline{g}_{i}, \overline{g}_{i}]$, then according to Definition \ref{dd8} and norm on $I(\mathbb{R})^{n}$, the  corresponding sequence $\left\{{\textbf{G}}_{ki}\right\}$ in $I(\mathbb{R})$ converges to ${\textbf{G}}_{i}\in I(\mathbb{R})$ for each $i=1,2,\ldots,n$. Also, by Definition \ref{dd9}, the sequences $\underline{g}_{ki}$ and $\overline{g}_{ki}$ in $\mathbb{R}$ converge to $\underline{g}_{i}$ and $\overline{g}_{i}$ in $\mathbb{R}$, respectively, for each $i=1,2,\ldots,n$. 
	\end{remark}

\subsection{Results from convex analysis}
	Apart from the results of interval analysis, we use the following results from classical convex analysis throughout the article.
	\begin{definition}\label{dd6}
(\emph{Projection} \cite{rockafellar2009variational}). Let $A$ be a nonempty closed set in $\mathbb{R}^{n}$. Then, the projection of a point $x\in \mathbb{R}^{n}$ onto the set A is denoted by $P(x~|~A)$, and is defined by
$$P(x~|~A)=\{y\in A :\lVert x-y\rVert = \inf\{\lVert x-u\rVert: u\in A\}\}.$$
\end{definition}
		\begin{definition} \label{g39}
		(\emph{Polar cone} \cite{rockafellar2009variational}). Let $A$  be a nonempty set in $\mathbb{R}^{n}$. Then, the polar cone of the set $A$ is
		\[A^{o}=\{x^{*}\in \mathbb{R}^{n}:\langle x^{*},x\rangle\leq 0~\text{ for all }x\in A\}.\]
		\end{definition}
		\begin{definition}
		(\emph{Tangent cone} \cite{rockafellar2009variational}). Let $A$ be a nonempty closed convex set in $\mathbb{R}^{n}$. Then, the tangent cone to the set $A$ at $x\in A$ is defined by
		\[T_{A}(x)=\cl\left(\bigcup\limits_{t>0} \frac{ A-x } { t }  \right).\]
		\end{definition}
		\begin{definition} \label{d2}
		(\emph{Normal cone} \cite{rockafellar2009variational}). The normal cone to a nonempty set $A$ in $\mathbb{R}^{n}$ at $x$ is polar of the tangent cone at $x$ to the $A$, i.e.,  $N_{A}(x)={T_{A}(x)}^{o}$.
		Therefore,
		\[N_{A}(x)=\left\{x^{*}\in \mathbb{R}^{n}:\langle x^{*},y-x\rangle\leq 0, \text{ for any~} y\in A\right\}.\]
		\end{definition}

\begin{lemma}\label{kr1}\emph{\cite{hiriart2004fundamentals}} 
Consider a convex set $S\subseteq {\mathbb{R}}^n$. Then, $\bar{x}$ is an element in the closure of $S$ if and only if $\langle x,\bar{x}\rangle\leq \psi^{*}_{S}({x})$ for all $x\in {\mathbb{R}}^n,$ where $\psi^{*}_{S}$ is the support function of $S$, i.e., $\psi^{*}_{S}({x})=\sup\limits_{s\in S}\langle x,s\rangle$.
\end{lemma}
\begin{lemma}\label{aa}
\textnormal{\cite{burke2002weak}} Let $C$ be a nonempty closed convex subset of $\mathbb{R}^{n}$.
\begin{enumerate}[\normalfont(i)]
\item\label{dst3} For all $y\in \mathbb{R}^{n}$,  $ \dis(y,C)= \sup\limits_{x\in C} \dis(y,x+T_{C}(x))$, 
where the distance function is given by
$\dis(y,C)= \inf\limits_{\overline{x}\in C}\lVert y-\overline{x}\rVert$.
\item\label{dst4} Define $\rho(x)=\dis(x,C)$. Then, for all $x\in C$ and $d\in \mathbb{R}^{n}$,
\[ \rho_{\mathscr{D}}(x)(d)= \dis(d,T_{C}(x))= \psi^{*}_{\mathbb{B}\cap N_{C}(x)}(d).\]
Moreover, if $d\in N_{C}(x),$ then
$\rho_{\mathscr{D}}(x)(d)= \dis(d,T_{C}(x))= \psi^{*}_{\mathbb{B}\cap N_{C}(x)}(d)=\lVert d \rVert.$
\end{enumerate}
\end{lemma}
\section{Support function in \texorpdfstring{$I(\mathbb{R})^{n}$}{Lg}} \label{Sec3}

In this section, we attempt to extend the conventional notion of support functions for subsets of $I(\mathbb{R})^{n}$. The derived concepts of support function are used later in Section \ref{Sec4} to derive dual characterizations of WSM for convex IVFs.

\begin{definition}\label{dd4}
(\emph{Support function of a subset of  $I(\mathbb{R})^{n}$}). Let $ \textbf{S} \text{ be a nonempty subset of } I(\mathbb{R})^{n}$. Then, the support function of  $\textbf{S}$ at ${x}\in\mathbb{R}^{n}$, denoted by $\boldsymbol{\psi^{*}_{\textbf{S}}}({x})$, is defined by \[
\boldsymbol{\psi^{*}_{\textbf{S}}}({x})= \sup_{\widehat{\textbf{A}}\in \textbf{S}} {x}^{\top}\odot \widehat{\textbf{A}}. 
\]
 \end{definition}
\begin{lemma}\label{g40}
		Let $\textbf{S}_{1},\textbf{S}_{2}$ be two nonempty subsets of $I(\mathbb{{R}})^{n} $ such that $\textbf{S}_{1}\subseteq\textbf{S}_{2}$. Then, for any $x\in {X}\subseteq\mathbb{R}^{n}$,
	$$\mathbf{\boldsymbol{\psi}^{*}_{\textbf{S}_{1}}}(x)\preceq \boldsymbol{\psi^{*}_{\textbf{S}_{2}}}(x).$$
	\end{lemma}
	\begin{proof}  For any $\widehat{\textbf{B}}=\left([\underline{b}_{1},\overline{b}_{1}],[\underline{b}_{2},\overline{b}_{2}],\ldots,[\underline{b}_{n},\overline{b}_{n}]\right)\in \textbf{S}_{2}$ and $x\in X$, we have
$	{x}^{\top}\odot\widehat{\textbf{B}}\preceq\boldsymbol{\psi^{*}_{\textbf{S}_{2}}}(x).$
	Given $\textbf{S}_{1}\subseteq \textbf{S}_{2},$ i.e., for any  $\widehat{\textbf{D}}=\left([\underline{d}_{1},\overline{d}_{1}],[\underline{d}_{2},\overline{d}_{2}],\ldots,[\underline{d}_{n},\overline{d}_{n}]\right)\in \textbf{S}_{1}, \text{ we have } \widehat{\textbf{D}}\in\textbf{S}_{2}.$ Therefore, 
	\[ {x}^{\top}\odot\widehat{\textbf{D}}\preceq\boldsymbol{\psi^{*}_{\textbf{S}_{2}}}(x).\]
Since $\widehat{\textbf{D}}$ is arbitrary, we get
	\[
\boldsymbol{\psi^{*}_{\textbf{S}_{1}}}(x)=\sup_{\widehat{\textbf{E}}\in \textbf{S}_{1}} {x}^{\top}\odot\widehat{\textbf{E}}\preceq\boldsymbol{\psi^{*}_{\textbf{S}_{2}}}(x). 
\]	
\end{proof}
\begin{theorem} \label{lm1}
	Let $K$ be a nonempty closed convex cone in ${X}\subseteq\mathbb{R}^{n}$. Let $\textbf{P}$ and $\textbf{Q}$ be two nonempty subsets of $I(\mathbb{{R}})^{n}$. Then, 
\begin{eqnarray*}
\boldsymbol\psi^{*}_{\textbf{P}}(x)\preceq\boldsymbol\psi^{*}_{\textbf{Q}}(x) ~\text{for all}~ x\in K &\text{if and only if}& \boldsymbol\psi^{*}_{\textbf{P}}(x)\preceq\boldsymbol\psi^{*}_{\textbf{Q}\oplus K^{o}}(x)~\text{for all} ~x\in {X},
\end{eqnarray*}
where $K^{o}$ is the polar cone of $K$.
\end{theorem}
\begin{proof} 
Let $\boldsymbol{\psi^{*}_{\textbf{P}}}(x)\preceq\boldsymbol{\psi^{*}_{\textbf{Q}}}(x) ~\text{for all}~ x\in K$.
Consider $x\in K$. Clearly $\textbf{Q}\subseteq\textbf{Q}\oplus K^{o}$. Then, by Lemma \ref{g40}, we have 
\begin{align}
\boldsymbol\psi^{*}_\textbf{Q}(x)\preceq&~\boldsymbol\psi^{*}_{\textbf{Q}\oplus K^{o}}(x)\nonumber\\\preceq&~\boldsymbol{\psi^{*}_{\textbf{Q}}}(x) \oplus\psi^{*}_{K^{o}}(x)\text{ by (\ref{g77})} \text{ of Lemma \ref{g79}} \text{ and Definition }\ref{dd4}\nonumber\\\preceq&~ \boldsymbol{\psi^{*}_{\textbf{Q}}}(x)~ \text{because}~\psi^{*}_{K^{o}}(x)=0 .\nonumber
\end{align}
\text{Therefore,} 
\begin{align}\label{eq52}
\boldsymbol{\psi^{*}_\textbf{Q}}(x) =&~\boldsymbol\psi^{*}_{\textbf{Q}\oplus K^{o}}(x)~ \text{for all}~ x\in K.
\end{align}
Also, by hypothesis, we have $\boldsymbol{\psi^{*}_{\textbf{P}}}(x)\preceq\boldsymbol{\psi^{*}_{\textbf{Q}}}(x) ~\text{for all}~ x\in K$, and hence
\begin{equation}\label{kr3}
    \psi^{*}_{\textbf{P}}(x)\preceq\boldsymbol\psi^{*}_{\textbf{Q}\oplus K^{o}}(x)~\text{for all} ~x\in {K}.
\end{equation}
Suppose now if $x\notin K$, then there exists $z\in K^{o}$ such that $\langle z,x\rangle>0$.
Thus, for any  $\widehat{\textbf{A}}\in\textbf{Q}$ and  $\lambda\geq 0,~ \widehat{\textbf{A}}\oplus\lambda z \in\textbf{Q}\oplus K^{o}$. Also, ${x}^{\top}\odot(\widehat{\textbf{A}}\oplus \lambda z)$
\begin{eqnarray*}
&=&\Bigg[\min\left\{\sum\limits_{i=1}^{n} x_{i}(\underline{a}_{i}+\lambda z_i),\sum\limits_{i=1}^{n} x_{i}(\overline{a}_{i}+\lambda z_i)\right\},
\max\left\{\sum\limits_{i=1}^{n} x_{i}(\underline{a}_{i}+\lambda z_i),\sum\limits_{i=1}^{n} x_{i}(\overline{a}_{i}+\lambda z_i)\right\}\Bigg]\\
&=&\Bigg[\min\left\{\sum\limits_{i=1}^{n} x_{i}\underline{a}_{i}+\lambda {x}^{\top}z,\sum\limits_{i=1}^{n} x_{i}\overline{a}_{i}+\lambda {x}^{\top}z\right\},\max\left\{\sum\limits_{i=1}^{n} x_{i}\underline{a}_{i}+\lambda {x}^{\top}z,\sum\limits_{i=1}^{n} x_{i}\overline{a}_{i}+\lambda {x}^{\top}z\right\}\Bigg].
\end{eqnarray*}
Note that as $\lambda\rightarrow +\infty,~\lambda {x}^{\top} z\rightarrow +\infty$, and   therefore ${x}^{\top}\odot(\widehat{\textbf{A}}\oplus\lambda z)\rightarrow +\infty$, 
which implies 
\[ \boldsymbol\psi^{*}_{\textbf{Q}\oplus K^{o}}(x)= [+\infty,+\infty].
\]
Thus, 
\begin{equation}\label{kr4}
    \boldsymbol\psi^{*}_{\textbf{P}}(x)\preceq\boldsymbol\psi^{*}_{\textbf{Q}\oplus K^{o}}(x)~\text{for all} ~x\in X\backslash K.
\end{equation}
Therefore, from (\ref{kr3}) and (\ref{kr4}), we have
$$\boldsymbol\psi^{*}_{\textbf{P}}(x)\preceq\boldsymbol\psi^{*}_{\textbf{Q}\oplus K^{o}}(x)~\text{for all} ~x\in {X}.$$
Proof of the converse part follows from (\ref{eq52}). This completes the proof.
\end{proof}
%
%
%
%
%
%
%
%
\begin{lemma}\label{lm5}
Let P be nonempty a subset of $\mathbb{R}^{n}$ and \textbf{Q} be a nonempty closed convex subset of $I(\mathbb{R})^{n}$. Then,  for any $ x\in \mathbb{R}^{n}$,
\[
\psi^{*}_{P}(x)\preceq\boldsymbol\psi^{*}_{\textbf{Q}}(x)~\text{if and only if}~ P\subseteq\textbf{Q}.
\]
\end{lemma}

\begin{proof}
Let $\psi^{*}_{P}(x)\preceq\boldsymbol\psi^{*}_{\textbf{Q}}(x)$ for  $x\in\mathbb{R}^{n}$.
Therefore, for any $p\in P$ and $x\in\mathbb{R}^{n}$, we have 
\begin{eqnarray*}
&&\langle x,p\rangle\preceq \sup_{\widehat{\textbf{Q}}_{i}\in \textbf{Q}} {x}^{\top}\odot\widehat{\textbf{Q}}_{i},~\text{where}~\widehat{\textbf{Q}}_{i}=\left(\left[\underline{q}_{i1},\overline{q}_{i1}\right],\left[\underline{q}_{i2},\overline{q}_{i2}\right],\ldots,\left[\underline{q}_{in},\overline{q}_{in}\right]\right)\\&\implies&
\langle x,p\rangle\preceq\sup\limits_{\widehat{\textbf{Q}}_{i}\in\textbf{Q}}\left[\min\left\{\sum\limits_{j=1}^{n} x_{j}\underline{q}_{ij},\sum\limits_{j=1}^{n} x_{j}\overline{q}_{ij}\right\},\max\left\{\sum\limits_{j=1}^{n} x_{j}\underline{q}_{ij},\sum\limits_{j=1}^{n} x_{j}\overline{q}_{ij}\right\}\right].
\end{eqnarray*}
We now consider the following two possible cases. 
\begin{enumerate}[$\bullet$ \text{Case} 1.]
    \item\label{kr2} Let $\sum\limits_{j=1}^{n} x_{j}\underline{q}_{ij}\leq\sum\limits_{j=1}^{n} x_{j}\overline{q}_{ij}$.
    In this case, we have
\begin{equation}\label{eq64}
\langle x,p\rangle\preceq\sup\limits_{\widehat{\textbf{Q}}_{i}\in\textbf{Q}}\left[\sum\limits_{j=1}^{n} x_{j}\underline{q}_{ij},\sum\limits_{j=1}^{n} x_{j}\overline{q}_{ij}\right].
\end{equation}

    Next, define two sets $S_{1}$ and $S_{2}$ such that $S_{1}=\left\{\underline{Q}_{1}, \underline{Q}_{2},\ldots,\underline{Q}_{n},\ldots\right\}$ and $S_{2}\\=\left\{\overline{Q}_{1}, \overline{Q}_{2},\ldots,\overline{Q}_{n},\ldots\right\}$, where $\underline{Q}_{i}=\left(\underline{q}_{i1},\underline{q}_{i2},\ldots,\underline{q}_{in}\right)\in\mathbb{R}^{n}$ and $\overline{Q}_{i}=\\\left(\overline{q}_{i1},\overline{q}_{i2},\ldots,\overline{q}_{in}\right)\in\mathbb{R}^{n}$.\\
    Therefore, (\ref{eq64}) along with Remark \ref{nt5} gives,
    \begin{align}\label{eq65}
        \langle x,p \rangle&\leq\sup\limits_{\underline{Q}_{i}\in S_{1}}\left\langle x,\underline{Q}_{i}\right\rangle\\
         \text{ and } \langle x,p\rangle&\leq\sup\limits_{\overline{Q}_{i}\in S_{2}}\left\langle x,\overline{Q}_{i}\right\rangle.
    \end{align}
     Thus, from (\ref{eq65}) and Lemma \ref{kr1}, we have $ p\in S_{1}, \text{ i.e., }  p=\underline{Q}_{m} \text{ for some } m$. \\ 
      To show that $p\in\textbf{Q}$, we have to show that $p=\overline{Q}_{m}$ as well. \\
     Note that 
     \begin{align}\label{eq66}
     \langle x,p\rangle& =\left\langle x, \underline{Q}_{m}\right\rangle\leq \left\langle x,\overline{Q}_{m}\right\rangle \text{ for all } x\in\mathbb{R}^{n}~\text{because}~ \sum\limits_{j=1}^{n} x_{j}\underline{q}_{ij}\leq\sum\limits_{j=1}^{n} x_{j}\overline{q}_{ij}\nonumber\\ \implies \langle x,p\rangle&\leq\sup\left\langle x,\overline{Q}_{m}\right\rangle \text{ for all } x\in\mathbb{R}^{n}\nonumber\\\implies
     \langle x,p\rangle&\leq\psi^{*}_{S^{'}}(x)\text{ for all } x\in\mathbb{R}^{n}, \text{ where } S' ~\text{is the singleton set}~ \left\{\overline{Q}_{m}\right\}.
     \end{align}
     Thus, from equation (\ref{eq66}) and Lemma \ref{kr1}, we have $p\in S^{'},$ i.e., $p=\overline{Q}_{m}$. \\Hence,  $p\in\textbf{Q}.$ Since $p$ is arbitrary, $P\subseteq{\textbf{Q}}$.
     \item Let $\sum\limits_{j=1}^{n} x_{j}\overline{q}_{ij}\leq\sum\limits_{j=1}^{n} x_{j}\underline{q}_{ij}$.
By following similar steps as in Case \ref{kr2}, in this case also, we get $P\subseteq{\textbf{Q}}$. 
\end{enumerate}
Proof of the converse part follows from Lemma \ref{g40}.
\end{proof}
 \begin{lemma}\label{lm9}
            For $x\in\mathbb{R}^{n}$ and $\widehat{\textbf{A}}=\left(\textbf{A}_{1},\textbf{A}_{2},\ldots,\textbf{A}_{n}\right)\in\textbf{S}\subseteq\mathbb{R}^{n},$ we have
            \[
            {x}^{\top}\odot\widehat{\textbf{A}}\preceq \lVert x \rVert\odot\left[\lVert \widehat {\textbf{A}} \rVert_{I(\mathbb{R})^{n}}, \lVert \widehat {\textbf{A}} \rVert_{I(\mathbb{R})^{n}} \right].
            \]
    \end{lemma}
    
    \begin{proof}
   
    Note that   ${x}^{\top}\odot\widehat{\textbf{A}}$
    \begin{eqnarray*}
    &=&\Bigg[\min\left\{\sum\limits_{i=1}^{n} x_{i}\underline{a}_{i}, \sum\limits_{i=1}^{n} x_{i}\overline{a}_{i}\right\},\max\left\{\sum\limits_{i=1}^{n} x_{i}\underline{a}_{i}, \sum\limits_{i=1}^{n} x_{i}\overline{a}_{i}\right\}\Bigg]\\
&\preceq&\Bigg[ \min \left \{\sum \limits_{i=1}^{n} \lvert x_{i} \rvert \lVert \textbf{A}_{i} \rVert_{I(\mathbb{R})} , \sum\limits_{i=1}^{n} \lvert x_{i}\rvert \lVert \textbf{A}_{i} \rVert_{I(\mathbb{R})} \right\},\max\left\{\sum\limits_{i=1}^{n} \lvert x_{i} \rvert \lVert \textbf{A}_{i} \rVert_{I(\mathbb{R})}, \sum\limits_{i=1}^{n} \lvert x_{i}\rvert \lVert \textbf{A}_{i} \rVert_{I(\mathbb{R})} \right\}\Bigg]\\
&=&\left[ \min \left \{\lVert x \rVert \lVert \widehat {\textbf{A}} \rVert_{I(\mathbb{R})^{n}} , \lVert x \rVert \lVert \widehat {\textbf{A}} \rVert_{I(\mathbb{R})^{n}} \right\},\max\left \{\lVert x \rVert \lVert \widehat {\textbf{A}} \rVert_{I(\mathbb{R})^{n}}, \lVert x \rVert \lVert \widehat {\textbf{A}} \rVert_{I(\mathbb{R})^{n}} \right\}\right]\\
&\preceq&  \lVert x \rVert\odot\left[\lVert \widehat {\textbf{A}} \rVert_{I(\mathbb{R})^{n}}, \lVert \widehat {\textbf{A}} \rVert_{I(\mathbb{R})^{n}} \right].
    \end{eqnarray*}
    \end{proof}
 
\begin{lemma}\label{lm11}
        The support function of a nonempty set $\textbf{S}\subseteq I(\mathbb{R})^{n}$ is finite everywhere if and only if $\textbf{S}$ is bounded.
\end{lemma}
\begin{proof}
Suppose that $\textbf{S}$ is bounded, i.e., we have $M>0$ such that 
$
\lVert \widehat {\textbf{A}} \rVert_{I(\mathbb{R})^{n}}\leq M \text{ for all } \widehat {\textbf{A}}=(\textbf{A}_{1},\textbf{A}_{2},\ldots,\textbf{A}_{n})\in\textbf{S} \text{ with } \textbf{A}_{i}=[\underline{a}_{i},\overline{a}_{i}] \text{ for each } i=1,2,\ldots,n.$
By Lemma \ref{lm9}  and $\lVert \widehat {\textbf{A}} \rVert_{I(\mathbb{R})^{n}}\\\leq$ $M$, for any $x\in\mathbb{R}^{n},$ we have 
\begin{equation*}
{x}^{\top}\odot\widehat{\textbf{A}} \preceq \lVert x \rVert\odot\left[\lVert \widehat {\textbf{A}} \rVert_{I(\mathbb{R})^{n}}, \lVert \widehat {\textbf{A}} \rVert_{I(\mathbb{R})^{n}} \right] \preceq \lVert x \rVert\odot\left[M, M \right] 
\preceq \lVert x \rVert M.
\end{equation*}
Since $\widehat{\textbf{A}}\in\textbf{S}$ is arbitrary chosen, therefore 
\begin{eqnarray*}
\boldsymbol{\psi^{*}_{\textbf{S}}}({x})&=& \sup_{\widehat{\textbf{A}}\in \textbf{S}} {x}^{\top}\odot \widehat{\textbf{A}}
\preceq\lVert x \rVert M.
\end{eqnarray*}
Hence, $\boldsymbol{\psi^{*}_{\textbf{S}}}({x})$ is finite everywhere.\newline
Conversely,  let $\boldsymbol{\psi^{*}_{\textbf{S}}}({x})$ is finite for every $x\in\mathbb{R}^{n}.$ Therefore, there exists an $M>0$ such that $
\boldsymbol{\psi^{*}_{\textbf{S}}}({x})\preceq M,$
which implies that for any $x\in\mathbb{R}^{n}$ and  ${\widehat{\textbf{A}}\in \textbf{S}}$, we have
\begin{eqnarray*}
 &&{x}^{\top}\odot\widehat{\textbf{A}}=\Bigg[\min\left\{\sum\limits_{i=1}^{n} x_{i}\underline{a}_{i}, \sum\limits_{i=1}^{n} x_{i}\overline{a}_{i}\right\},\max\left\{\sum\limits_{i=1}^{n} x_{i}\underline{a}_{i}, \sum\limits_{i=1}^{n} x_{i}\overline{a}_{i}\right\}\Bigg]\preceq M\\
 &\implies& \sum\limits_{i=1}^{n} x_{i}\underline{a}_{i}\leq M \text{ and }  \sum\limits_{i=1}^{n} x_{i}\overline{a}_{i}\leq M.
\end{eqnarray*}
Take $\sum\limits_{i=1}^{n} x_{i}\underline{a}_{i}\leq M$, then by Remark \ref{n1}, we have
\begin{eqnarray}\label{eq33}
\langle x, \underline{a}\rangle\leq M, \text{ where } \underline{a}=(\underline{a}_{1},\underline{a}_{2},\ldots,\underline{a}_{n})\in\mathbb{R}^{n}.
\end{eqnarray}
 If $\underline{a}\neq0$, choose $x=\frac{\underline{a}}{\lVert\underline{a}\rVert}$, then  (\ref{eq33}) gives
\begin{eqnarray*}
&&\left\langle \frac{\underline{a}}{\lVert\underline{a}\rVert}, \underline{a} \right\rangle \leq M\\
&\implies& \lVert\underline{a}\rVert\leq M, \text{ where } \underline{a}=(\underline{a}_{1},\underline{a}_{2},\ldots,\underline{a}_{n})\in\mathbb{R}^{n}\\ 
&\implies& \lvert \underline{a}_{i}\rvert \leq M \text{ for each } i=1,2,\ldots,n.
\end{eqnarray*}
Similarly, when we take $\sum\limits_{i=1}^{n} x_{i}\overline{a}_{i}\leq M$, we get $\lvert \overline{a}_{i} \rvert \leq M$ for each $i=1,2,\ldots,n$. Therefore, we have 
\begin{eqnarray*}
&&\textbf{A}_{i}=[\underline{a}_{i},\overline{a}_{i}]\preceq M \text{ for each } i=1,2,\ldots,n\\ 
&\implies& \widehat{\textbf{A}}\preceq M.
\end{eqnarray*}
Since $\widehat{\textbf{A}}\in\textbf{S}$ was arbitrary chosen, therefore we have
$
\widehat{\textbf{A}}\preceq M \text{ for all } \widehat{\textbf{A}}\in\textbf{S}. 
$
Hence, $\textbf{S}$ is bounded.
\end{proof}
\section{\texorpdfstring{$gH$}{Lg}-subdifferentiability of convex IVFs}\label{sec4_new}
In this section we develop $gH$-subdifferential calculus for convex IVFs that are used later to find dual characterization of WSM for convex IVFs.

	\begin{definition}\label{dd5}(\emph{$gH$-subdifferentiability}). Let $\textbf{F}:{X}\subseteq {\mathbb{R}}^n\rightarrow \overline{I(\mathbb{R})}$ be a proper convex IVF and  $\bar{x}\in \dom(\textbf{F})$. Then, $gH$-subdifferential of $\textbf{F}$ at $\bar{x}$, denoted by $  \mathbf{\partial} \textbf{F}(\bar{x})$ is defined by
	\begin{equation}\label{kg1}
	   \boldsymbol \partial \textbf{F}(\bar{x})=\left\{\widehat{\textbf{G}}\in I(\mathbb{R})^n:	(x-\bar{x})^{\top}\odot\widehat{\textbf{G}}\preceq\textbf{F}(x)\ominus_{gH}\textbf{F}(\bar{x}) \text{ for all}~x\in X\right\}.
	\end{equation}
	The elements of (\ref{kg1}) are known as $gH$-subgradients of $\textbf{F}$ at $\bar{x}$. Further, if  $\boldsymbol\partial \textbf{F}(\bar{x})\neq\emptyset$, we say that $\textbf{F}$ is $gH$-subdifferentiable at $\bar{x}.$
		\end{definition}
		\begin{example}\label{ex1}
		Consider $\textbf{F}:\mathbb{R}\rightarrow I(\mathbb{R})$ be a convex IVF such that $\textbf{F}(x)=\lvert x \rvert\odot\textbf{A},$ where $\textbf{0}\preceq\textbf{A}.$ Let us check $gH$-subdifferentiability of $\textbf{F}$ at $0$.
		\begin{eqnarray}
\boldsymbol\partial\textbf{F}({0})&=&\left\{{\textbf{G}}\in I(\mathbb{R}):	(x-0)\odot{\textbf{G}}\preceq\textbf{F}(x)\ominus_{gH}\textbf{F}(0) \text{ for all }x\in \mathbb{R}\right\}\nonumber\\&=&\left\{{\textbf{G}}\in I(\mathbb{R}):	x\odot\textbf{G}\preceq\lvert x\rvert\odot\textbf{A} \text{ for all } x\in \mathbb{R}\right\}\label{kr11}
\end{eqnarray}
	
		\begin{enumerate}[$\bullet$ \text{Case} 1.]
		    \item\label{13} $x\leq0$. In this case, for all $x\in\mathbb{R}$,  $(\ref{kr11})$ gives,
		    \begin{eqnarray*}
		    &&x\odot\textbf{G}\preceq (-x)\odot\textbf{A}
		    \implies (-1)\odot\textbf{A}\preceq\textbf{G}.
		    \end{eqnarray*}
	\item\label{kr41} $x>0$. In this case, for all $x\in\mathbb{R}$, (\ref{kr11}) gives,  
	\begin{eqnarray*}
		    &&x\odot\textbf{G}\preceq x\odot\textbf{A}
		    \implies\textbf{G}\preceq\textbf{A}.
		    \end{eqnarray*}
		\end{enumerate}
		Hence, from Case \ref{13} and Case \ref{kr41}, we have 
	$	\boldsymbol \partial \textbf{F}(0)=\{\textbf{G}\in I(\mathbb{R}) : (-1)\odot\textbf{A}\preceq\textbf{G}\preceq\textbf{A}$\}.
	\end{example}
\begin{figure}[H]
\centering
\begin{tikzpicture}[scale=2]
  \draw[->,thick] (-1.3, 0) -- (1.7, 0);
  \node  at (1.8, 0) {$X$}; 
  \draw[->,thick] (0, -1.3) -- (0, 1.7);
  \node  at (0, 1.8) {$F$}; 
  \node at (-0.085, -0.15) {$O$};
  \draw[dashed,color=red,thick]  (0, 0) -- (1.5, 1.5);
  \draw[dashed,color=red,thick]  (0, 0) -- (1.5, 0.75);
  \draw[dashed,color=red,thick]  (0, 0) -- (-1.5, 1.5);
  \draw[dashed,color=red,thick]  (0, 0) -- (-1.5, 0.75);
 \path[fill=lightgray] (0,0) -- (1.5,0.75) -- (1.5,1.5) -- cycle;
 \path[fill=lightgray] (0,0) -- (-1.5,0.75) -- (-1.5,1.5) -- cycle;
 \path[fill=mylightblue] (0,0) -- (1.5,0.9375) -- (1.5,1.3125) -- cycle;
 \path[fill=mylightblue] (0,0) -- (-1.5,-0.9375) -- (-1.5,-1.3125) -- cycle;
 \path[fill=mylightgreen] (0,0) -- (-1.5,0.9375) -- (-1.5,1.3125) -- cycle;
 \path[fill=mylightgreen] (0,0) -- (1.5,-0.9375) -- (1.5,-1.3125) -- cycle;
 
 \draw (0.5,1) node {$\mathbf{G}_2$}; 
 \draw[->] (0.6, 0.9) -- (0.9, 0.7); 
 
 \draw (-0.5,1) node {$\mathbf{G}_1$}; 
 \draw[->] (-0.6, 0.9) -- (-0.9, 0.7);
 
 \draw (-0.5,-1) node {$\mathbf{G}_2$}; 
 \draw[->] (-0.6, -0.9) -- (-0.9, -0.7);
 
 \draw (0.5,-1) node {$\mathbf{G}_1$}; 
 \draw[->] (0.6, -0.9) -- (0.9, -0.7);
  \end{tikzpicture}
  \caption{The IVF $\textbf{F}$ of Example \ref{ex1}}
    \label{Fig1}
\end{figure}

	In Fig. \ref{Fig1}, the IVF $\mathbf{F}$, with $\mathbf{A} = \left[\tfrac{1}{4},1\right]$, is drawn by the gray shaded region between two red dashed lines, and its possible two $gH$-subgradients $\textbf{G}_{1}$ and $\textbf{G}_{2}$ at $0$ are shown by blue and green shaded regions, respectively.

\begin{lemma}\label{lm7} Let $X$ be a nonempty convex subset of ${\mathbb{R}}^n$ and $\textbf{F}:X\rightarrow {\overline{I(\mathbb{R})}}$ be a proper convex IVF. Then, for any $\bar x\in\dom(\textbf{F})$ and $h\in\mathbb{R}^{n}$ such that $\bar{x}+h\in X$, the $gH$-subdifferential set of \textbf{F} at $\bar{x}$ is
		\[
		\boldsymbol \partial \textbf{F}(\bar{x})=\left\{ \widehat{\textbf{G}}\in I(\mathbb{R})^{n}: h^{\top}\odot\widehat{\textbf{G}}\preceq\textbf{F}_\mathscr{D}(\bar{x})(h)\right\}
		,\]
		
		where $\textbf{F}_\mathscr{D}(\bar{x})(h)$ is $gH$-directional derivative of \textbf{F} at $\bar x$ in the direction of $h$.
		\end{lemma}
		\begin{proof}
		Suppose $\widehat{\textbf{G}}\in\boldsymbol \partial \textbf{F}(\bar{x}).$ Then, by Definition \ref{dd5}, we have 
		\begin{equation}\label{kg2}
		    (x-\bar{x})^{\top}\odot\widehat{\textbf{G}}\preceq\textbf{F}(x)\ominus_{gH}\textbf{F}(\bar{x}) \text{ for all } x\in X.
		\end{equation}
		By taking $x=\bar{x}+\lambda h ~\text{with}~ \lambda>0$ and $h\in \mathbb{R}^{n}$ in (\ref{kg2}), we get
		\begin{eqnarray*}
		&& 	 h^{\top}\odot\widehat{\textbf{G}}\preceq\frac{\textbf{F}(\bar{x}+\lambda h)\ominus_{gH}\textbf{F}(\bar{x})}{\lambda}\\&\implies&
		h^{\top}\odot\widehat{\textbf{G}}\preceq\lim\limits_{\lambda\rightarrow 0}\frac{\textbf{F}(\bar{x}+\lambda h)\ominus_{gH}\textbf{F}(\bar{x})}{\lambda}\\&\implies&
		 h^{\top}\odot\widehat{\textbf{G}}\preceq\textbf{F}_{\mathscr{D}}(\bar{x})(h).
		\end{eqnarray*}
		Next, if we take any $\widehat{\textbf{G}}\in I(\mathbb{R})^{n}$ such that  $h^{\top}\odot\widehat{\textbf{G}}\preceq\textbf{F}_\mathscr{D}(\bar{x})(h)~ \text{for all}~  h\in \mathbb{R}^{n}$. Then, by a similar reasoning as above it can be seen that $\widehat{\textbf{G}}\in \boldsymbol \partial \textbf{F}(\bar{x})$.
		
		\end{proof}
	\begin{theorem}\label{kg9}
		Let $X$ be a nonempty convex subset of ${\mathbb{R}}^n$ and $\textbf{F}: X \rightarrow {\overline{I(\mathbb{R})}}$ be a proper convex IVF with $\textbf{F}(x)=[\underline{F}(x),\overline{F}(x)]$, where $\underline{F},~ \overline{F}:X\rightarrow \overline{\mathbb{R}}$ are extended real-valued functions. Then, for any $\bar{x}\in \dom(\textbf{F})$,~$\boldsymbol \partial \textbf{F}(\bar{x})$ is closed and convex.
	\end{theorem}
	\begin{proof}
	We first prove the closedness of $\boldsymbol \partial \textbf{F}(\bar{x})$.
		Let $\left\{\widehat{\textbf{G}}_{k}\right\}$ be a sequence in $\boldsymbol \partial \textbf{F}(\bar{x})$, which converges to $\widehat{\textbf{G}}\in I(\mathbb{R})^{n}$, where ${\widehat{\textbf{G}}_{k}}=(\textbf{G}_{k1},\textbf{G}_{k2},\ldots,\textbf{G}_{kn})$ and $\widehat{\textbf{G}}=(\textbf{G}_{1},\textbf{G}_{2},\ldots,\textbf{G}_{n})$.
		Since ${\widehat{\textbf{G}}_{k}}\in\boldsymbol \partial \textbf{F}(\bar{x})$, for all $h\in \mathbb{R}^{n}$ such that $\bar x+h\in X$, we have
		\begin{align}
    	&h^{\top}\odot\widehat{\textbf{G}}_{k}\preceq\textbf{F}(\bar{x}+h)\ominus_{gH}\textbf{F}(\bar{x}),\nonumber\\\implies&\min\left\{\sum\limits_{i=1}^{n} h_{i}\underline{g}_{ki},\sum\limits_{i=1}^{n} h_{i}\overline{g}_{ki}\right\}\leq\min\Big\{\underline{F}(\bar{x}+h)-\underline{F}(\bar{x}),\overline{F}(\bar{x}+h)-\overline{F}(\bar{x})\Big\}\nonumber\\\label{eq60}
    	\text{ and }&\max\left\{\sum\limits_{i=1}^{n} h_{i}\underline{g}_{ki},\sum\limits_{i=1}^{n} h_{i}\overline{g}_{ki}\right\}\leq\max\Big\{\underline{F}(\bar{x}+h)-\underline{F}(\bar{x}),\overline{F}(\bar{x}+h)-\overline{F}(\bar{x})\Big\}.
		\end{align}
		Since the sequence $\left\{\widehat{\textbf{G}}_{k}\right\}$ converges to $\widehat{\textbf{G}}$, in view of Remark \ref{nt1}, the sequences $
		\{\underline{g}_{ki}\}$ and $\left\{\overline{g}_{ki}\right\}$ converge to $\underline{g}_{i}$ and $\overline{g}_{i}$, respectively, for each $i=1,2,\ldots,n$. Thus, 
		\begin{equation}\label{eq58}
		    \sum\limits_{i=1}^{n} h_{i}\underline{g}_{ki}\rightarrow\sum\limits_{i=1}^{n} h_{i}\underline{g}_{i} \text{ and } \sum\limits_{i=1}^{n} h_{i}\overline{g}_{ki}\rightarrow\sum\limits_{i=1}^{n} h_{i}\overline{g}_{i} \text{ as } k\rightarrow\infty.
		\end{equation}
		Therefore, in view of (\ref{eq60}) and (\ref{eq58}), we have
			\begin{eqnarray*}
		&&\left(\min\left\{\sum\limits_{i=1}^{n} h_{i}\underline{g}_{ki},\sum\limits_{i=1}^{n} h_{i}\overline{g}_{ki}\right\}\right)\rightarrow\left(\min\left\{\sum\limits_{i=1}^{n} h_{i}\underline{g}_{i},\sum\limits_{i=1}^{n} h_{i}\overline{g}_{i}\right\}\right)\\&&~~~~~~~~~~~~~~~~~~~~~~~~~~~~~~~~~~~~~~~~~~~~\leq\min\Big\{\underline{F}(\bar{x}+h)-\underline{F}(\bar{x}),\overline{F}(\bar{x}+h)-\overline{F}(\bar{x})\Big\}
		\end{eqnarray*}
		\text{ and }
		\begin{eqnarray*}
		&&\left(\max\left\{\sum\limits_{i=1}^{n} h_{i}\underline{g}_{ki},\sum\limits_{i=1}^{n} h_{i}\overline{g}_{ki}\right\}\right)\rightarrow\left(\max\left\{\sum\limits_{i=1}^{n} h_{i}\underline{g}_{i},\sum\limits_{i=1}^{n} h_{i}\overline{g}_{i}\right\}\right)\\&&~~~~~~~~~~~~~~~~~~~~~~~~~~~~~~~~~~~~~~~~~~~~\leq\max\Big\{\underline{F}(\bar{x}+h)-\underline{F}(\bar{x}),\overline{F}(\bar{x}+h)-\overline{F}(\bar{x})\Big\}.
		\end{eqnarray*}
		Thus,
		\begin{eqnarray*}
		&&\left[\min\left\{\sum\limits_{i=1}^{n} h_{i}\underline{g}_{i},\sum\limits_{i=1}^{n} h_{i}\overline{g}_{i}\right\},\max\left\{\sum\limits_{i=1}^{n} h_{i}\underline{g}_{i},\sum\limits_{i=1}^{n} h_{i}\overline{g}_{i}\right\}\right]\preceq\textbf{F}(\bar{x}+h)\ominus_{gH}\textbf{F}(\bar{x})\\&\implies&
	h^{\top}\odot\widehat{\textbf{G}}\preceq\textbf{F}(\bar{x}+h)\ominus_{gH}\textbf{F}(\bar{x}) \text{ for all } h\in X.
	\end{eqnarray*}
	Therefore, $\widehat{\textbf{G}}\in\boldsymbol \partial \textbf{F}(\bar{x})$, and hence $\boldsymbol \partial \textbf{F}(\bar{x})$ is closed.
	$\\$
To prove the convexity of $\boldsymbol \partial \textbf{F}(\bar{x})$, let 
$\widehat{\textbf{H}}=(\textbf{H}_{1},\textbf{H}_{2},\ldots,\textbf{H}_{n})$~and~$\widehat{\textbf{K}}=(\textbf{K}_{1},\textbf{K}_{2},\ldots,\textbf{K}_{n})$ be any two elements of~$\boldsymbol \partial \textbf{F}(\bar{x})$~with~$\textbf{H}_{i}$~=~$[\underline{h}_{i},\overline{h}_{i}]$~and~$\textbf{K}_{i}$~=~$[\underline{k}_{i},\overline {k}_{i}]$ for each $i=1,2,\ldots,n$. Then, for all $\lambda_{1},\lambda_{2}\geq0$, with $\lambda_{1}+\lambda_{2}=1$ and  for any $d\in \mathbb{R}^{n}$, we have
		\begin{eqnarray*}
		d^{\top}\odot\left(\lambda_{1}\odot\widehat{\textbf{H}}\oplus\lambda_{2}\odot\widehat{\textbf{K}}\right)&=&\Bigg[\min\left\{\sum\limits_{i=1}^{n} d_{i}( \lambda_{1}\underline{h}_{i}+\lambda_{2}\underline{k}_{i}), \sum\limits_{i=1}^{n} d_{i}( \lambda_{1}\overline{h}_{i}+\lambda_{2}\overline{k}_{i})\right\}, \\&&\max\left\{\sum\limits_{i=1}^{n} d_{i}( \lambda_{1}\underline{h}_{i}+\lambda_{2}\underline{k}_{i}), \sum\limits_{i=1}^{n} d_{i}( \lambda_{1}\overline{h}_{i}+\lambda_{2}\overline{k}_{i})\right\}\Bigg].
		\end{eqnarray*}	
		\begin{enumerate}[$ \bullet $ \text{Case} 1. ]
		    \item\label{cs1} Let  $\min\left\{\sum\limits_{i=1}^{n} d_{i}( \lambda_{1}\underline{h}_{i}+\lambda_{2}\underline{k}_{i}), \sum\limits_{i=1}^{n} d_{i}( \lambda_{1}\overline{h}_{i}+\lambda_{2}\overline{k}_{i})\right\}=\sum\limits_{i=1}^{n} d_{i}( \lambda_{1}\underline{h}_{i}+\lambda_{2}\underline{k}_{i})$. Then,  
		    \begin{eqnarray*} d^{\top}\odot\left(\lambda_{1}\odot\widehat{\textbf{H}}\oplus\lambda_{2}\odot\widehat{\textbf{K}}\right) &=&\left[\sum\limits_{i=1}^{n} d_{i}( \lambda_{1}\underline{h}_{i}+\lambda_{2}\underline{k}_{i}), \sum\limits_{i=1}^{n} d_{i}( \lambda_{1}\overline{h}_{i}+\lambda_{2}\overline{k}_{i})\right]\\&=&\left[\sum\limits_{i=1}^{n}\lambda_{1} d_{i}\underline{h}_{i}, \sum\limits_{i=1}^{n}\lambda_{1} d_{i} \overline{h}_{i}\Big]\oplus\Big[\sum\limits_{i=1}^{n}\lambda_{2} d_{i}\underline{k}_{i}, \sum\limits_{i=1}^{n}\lambda_{2} d_{i} \overline{k}_{i}\right]\\&=&\lambda_{1}\odot d^{\top}\odot\widehat{\textbf{H}}\oplus\lambda_{2}\odot d^{\top}\odot\widehat{\textbf{K}}\\&\preceq& \lambda_{1}\odot\textbf{F}_{\mathscr{D}}(\bar{x})(d)\oplus\lambda_{2}\odot\textbf{F}_{\mathscr{D}}(\bar{x})(d) \text{ by Lemma } \ref{lm7}\\&=& \textbf{F}_{\mathscr{D}}(\bar{x})(d) \text{ for any } d\in \mathbb{R}^{n}.
		    \end{eqnarray*}
            Hence, $d^{\top}\odot(\lambda_{1}\odot\widehat{\textbf{H}}\oplus\lambda_{2}\odot\widehat{\textbf{K}})\preceq\textbf{F}_{\mathscr{D}}(\bar{x})(d)$ for any $d\in \mathbb{R}^{n}$. Therefore, by Lemma \ref{lm7}, $\lambda_{1}\odot\widehat{\textbf{H}}\oplus\lambda_{2}\odot\widehat{\textbf{K}}\in\boldsymbol \partial \textbf{F}(\bar{x}).$
            \item Let  $\min\left\{\sum\limits_{i=1}^{n} d_{i}( \lambda_{1}\underline{h}_{i}+\lambda_{2}\underline{k}_{i}), \sum\limits_{i=1}^{n} d_{i}( \lambda_{1}\overline{h}_{i}+\lambda_{2}\overline{k}_{i})\right\}= \sum\limits_{i=1}^{n} d_{i}( \lambda_{1}\overline{h}_{i}+\lambda_{2}\overline{k}_{i})$.
            Proof contains similar steps as in Case \ref{cs1}.
		    \end{enumerate}
		    Thus, for any $\bar{x}\in\dom(\textbf{F}),  \boldsymbol \partial \textbf{F}(\bar{x})$ is convex.
	\end{proof}
	\begin{theorem}\label{th5}
	Let $X$ be a nonempty convex subset of ${\mathbb{R}}^n$ and let $\textbf{F}: X\rightarrow I(\mathbb{R})$ be a $gH$-differentiable convex IVF at $\bar{x}\in X$. Then, 
	    \[
	    \boldsymbol \partial \textbf{F}(\bar{x})=\left\{\nabla\textbf{F}(\bar{x})\right\}.
	    \]
	\end{theorem}
	\begin{proof}
	Let $\widehat{\textbf{G}}\in\boldsymbol \partial \textbf{F}(\bar{x})$. Since $\textbf{F}$ is $gH$-differentiable at $\bar{x}$, with the help of Lemma \ref{lm8} and Lemma \ref{lm7}, we get
	\begin{align}\label{eq61}
	&h^{\top}\odot\widehat{\textbf{G}}\preceq\textbf{L}_{\bar{x}}(h) \text{ for all } h\in\mathbb{R}^{n}\nonumber\\
	\implies h^{\top}&\odot\widehat{\textbf{G}}\preceq\sum\limits_{i=1}^{n} h_{i}\odot D_{i}\textbf{F}(\bar{x}) \text{ by Theorem }\ref{th3}.
	\end{align}
	Replacing $h$ by $-h$ in (\ref{eq61}), we obtain
	\begin{eqnarray*}
	    &&(-h)^{\top}\odot\widehat{\textbf{G}}\preceq\sum\limits_{i=1}^{n} (-h_{i})\odot D_{i}\textbf{F}(\bar{x})
	    \end{eqnarray*}
	    \begin{equation}\label{eq62}
	    \implies \sum\limits_{i=1}^{n} h_{i}\odot D_{i}\textbf{F}(\bar{x})\preceq h^{\top}\odot\widehat{\textbf{G}} \text{ for all } h\in\mathbb{R}^{n}.
	    \end{equation}
	Thus, (\ref{eq61}) and (\ref{eq62}), simultaneously give
	\begin{equation}\label{eq63}
	 \sum\limits_{i=1}^{n} h_{i}\odot D_{i}\textbf{F}(\bar{x})= h^{\top}\odot\widehat{\textbf{G}} \text{ for all } h\in\mathbb{R}^{n}.
	 \end{equation}
	Therefore, for each $i\in\left\{1,2,\ldots,n\right\},$ by choosing $h=e_{i}$ in (\ref{eq63}), we have
	$
	D_{i}\textbf{F}(\bar{x})=\textbf{G}_{i}.
	$\\
	Hence,$
	\nabla\textbf{F}(\bar{x})=\widehat{\textbf{G}}.$
	Since $\widehat{\textbf{G}}\in\boldsymbol \partial \textbf{F}(\bar{x})$ is arbitrary, 
	$\boldsymbol \partial \textbf{F}(\bar{x})=\left\{\nabla\textbf{F}(\bar{x})\right\}$.
	\end{proof}
\begin{lemma}\label{lm6}
	Let $X$ be a nonempty convex subset of ${\mathbb{R}}^n$ and $\textbf{F}:X \rightarrow {\overline{I(\mathbb{R})}}$ be a proper convex IVF with $\textbf{F}(x)=[\underline{F}(x),\overline{F}(x)]$, where $\underline{F},~ \overline{F}:X\rightarrow \overline{\mathbb{R}}$ are extended real-valued functions. Then, the subdifferential set of $\textbf{F}$ at $\bar x\in\inte(\dom(\textbf{F}))$ can be obtained by the subdifferential sets of $\underline{F}$ and $\overline{F}$ at $\bar x$ and vice-versa. 
	\end{lemma}
   \begin{proof}
   Since $\textbf{F}$ is proper convex, with the help of Lemma \ref{lm2}, we note that 
   $\underline{F}$ and $\overline{F}$ are also convex. Therefore, by the property of  real-valued  proper convex functions, the subdifferential sets of $\underline{F}$ and $\overline{F}$ at $\bar x\in\inte(\dom(\textbf{F}))$ are nonempty (see \cite{beck2017first}).
Let $\underline{g}=(\underline{g}_{1},\underline{g}_{2},\ldots,\underline{g}_{n})\in\partial\underline{F}(\bar x)$ and $\overline{g}=(\overline{g}_{1},\overline{g}_{2},\ldots,\overline{g}_{n})\in\partial\overline{F}(\bar{x})$. Then, by Definition  \ref{dd5} of $gH$-subdifferentiability, for any $h\in \mathbb{R}^{n}$ such that $\bar{x}+h\in X$ , we have 
\begin{equation}\label{eq53}
h^{\top}\odot\underline{g}\leq\underline{F}(\bar{x}+h)-\underline{F}(\bar{x}) \text{ and } h^{\top}\odot\overline{g}\leq\overline{F}(\bar{x}+h)-\overline{F}(\bar{x}).    
\end{equation}
Note that $\textbf{F}(\bar{x}+h)\ominus_{gH}\textbf{F}(\bar{x})$
\begin{eqnarray*}
&=&\Big[\min\Big\{\underline{F}(\bar{x}+h)-\underline{F}(\bar{x}),\overline{F}(\bar{x}+h)-\overline{F}(\bar{x})\Big\},\max\Big\{\underline{F}(\bar{x}+h)-\underline{F}(\bar{x}),\overline{F}(\bar{x}+h)-\overline{F}(\bar{x})\Big\}\Big]\\&\implies&
\Big[\min\Big\{h^{\top}\odot\underline{g}, h^{\top}\odot\overline{g}\Big\},\max\Big\{h^{\top}\odot\underline{g}, h^{\top}\odot\overline{g}\Big\}\Big]\preceq\textbf{F}(\bar{x}+h)\ominus_{gH}\textbf{F}(\bar{x})\text{ by}~(\ref{eq53})\\&\implies& h^{\top}\odot\widehat{\textbf{G}}\preceq\textbf{F}(\bar{x}+h)\ominus_{gH}\textbf{F}(\bar{x}), \text{ where  } \widehat{\textbf{G}}=(\textbf{G}_{1},\textbf{G}_{2},\ldots,\textbf{G}_{n})~\text{with}~ \textbf{G}_{i}=[\underline{g}_{i},\overline{g}_{i}]\\&\implies& \widehat{\textbf{G}}\in\boldsymbol \partial \textbf{F}(\bar{x}). 
\end{eqnarray*}
Thus, for any $\underline{g}\in\partial\underline{F}(\bar x)$ and $\overline{g}\in\partial\overline{F}(\bar{x})$, we have the corresponding $\widehat{\textbf{G}}\in\boldsymbol \partial \textbf{F}(\bar{x})$.
$\\$
To prove the converse part, for any $\bar{x}\in\inte(\dom(\textbf{F}))$, take $\widehat{\textbf{G}}=(\textbf{G}_{1},\textbf{G}_{2},\ldots,\textbf{G}_{n})\in\boldsymbol \partial \textbf{F}(\bar{x})$ with $\textbf{G}_{i}=[\underline{g}_{i},\overline{g}_{i}]$, $i=1,2,\ldots,n$. Then, by Definition \ref{dd5} of $gH$-subdifferentiability, we have
\begin{eqnarray*}
&& h^{\top}\odot\widehat{\textbf{G}}\preceq\textbf{F}(\bar{x}+h)\ominus_{gH}\textbf{F}(\bar{x}) \text{ for all } h\in \mathbb{R}^{n} \text{ such that } \bar{x}+h\in X\\&\implies&\Bigg[\min\left\{\sum\limits_{i=1}^{n} h_{i}\underline{g}_{i}, \sum\limits_{i=1}^{n} h_{i}\overline{g}_{i}\right\},\max\left\{\sum\limits_{i=1}^{n} h_{i}\underline{g}_{i}, \sum\limits_{i=1}^{n} h_{i}\overline{g}_{i}\right\}\Bigg]\preceq\textbf{F}(\bar{x}+h)\ominus_{gH}\textbf{F}(\bar{x}).
\end{eqnarray*}
Therefore, 
\begin{equation}\label{eq54}
\min\left\{\sum\limits_{i=1}^{n} h_{i}\underline{g}_{i}, \sum\limits_{i=1}^{n} h_{i}\overline{g}_{i}\right\}\leq\min\left\{\underline{F}(\bar{x}+h)-\underline{F}(\bar{x}),\overline{F}(\bar{x}+h)-\overline{F}(\bar{x})\right\}
\end{equation}
and
\begin{equation}\label{eq55}
    \max\left\{\sum\limits_{i=1}^{n} h_{i}\underline{g}_{i}, \sum\limits_{i=1}^{n} h_{i}\overline{g}_{i}\right\}\leq\max\left\{\underline{F}(\bar{x}+h)-\underline{F}(\bar{x}),\overline{F}(\bar{x}+h)-\overline{F}(\bar{x})\right\}.
\end{equation}
We now consider the following two possible cases. 
\begin{enumerate}[$\bullet$ \text{Case} 1.]
    \item\label{kg3} Let $\min\left\{\sum\limits_{i=1}^{n} h_{i}\underline{g}_{i}, \sum\limits_{i=1}^{n} h_{i}\overline{g}_{i}\right\}=\sum\limits_{i=1}^{n} h_{i}\underline{g}_{i}$ and $\min\Big\{\underline{F}(\bar{x}+h)-\underline{F}(\bar{x}),\overline{F}(\bar{x}+h)-\overline{F}(\bar{x})\Big\}\\=\underline{F}(\bar{x}+h)-\underline{F}(\bar{x})$.
    In this case, by (\ref{eq54}) and (\ref{eq55}), we have
    \begin{eqnarray*}
    &&\sum\limits_{i=1}^{n} h_{i}\underline{g}_{i} \leq\underline{F}(\bar{x}+h)-\underline{F}(\bar{x}) \text{ and } \sum\limits_{i=1}^{n} h_{i}\overline{g}_{i}\leq\overline{F}(\bar{x}+h)-\overline{F}(\bar{x})\\
    &\implies&
    h^{\top}\odot\underline{g}\leq\underline{F}(\bar{x}+h)-\underline{F}(\bar{x}) \text{ and } h^{\top}\odot\overline{g}\leq\overline{F}(\bar{x}+h)-\overline{F}(\bar{x}),\\&&
    \text{where } \underline{g}=(\underline{g}_{1},\underline{g}_{2},\ldots,\underline{g}_{n})\in\mathbb{R}^{n} \text{ and } \overline{g}=(\overline{g}_{1},\overline{g}_{2},\ldots,\overline{g}_{n})\in\mathbb{R}^{n}.
    \end{eqnarray*}
    Thus, we get $\underline{g}\in\partial\underline{F}(\bar x)$ and $\overline{g}\in\partial\overline{F}(\bar{x})$, which are required.
    \item\label{kg4} Let $\min\left\{\sum\limits_{i=1}^{n} h_{i}\underline{g}_{i}, \sum\limits_{i=1}^{n} h_{i}\overline{g}_{i}\right\}=\sum\limits_{i=1}^{n} h_{i}\overline{g}_{i}$ and $\min\Big\{\underline{F}(\bar{x}+h)-\underline{F}(\bar{x}),\overline{F}(\bar{x}+h)-\overline{F}(\bar{x})\Big\}=\underline{F}(\bar{x}+h)-\underline{F}(\bar{x}).$
    Proof contains similar steps as in Case \ref{kg3}. 
\end{enumerate}
From Case \ref{kg3} and Case \ref{kg4}, it is clear that for any $\widehat{\textbf{G}}\in\boldsymbol \partial \textbf{F}(\bar{x})$, we can obtain the subgradients of $\overline{F}$ and $\underline{F}$ at $\bar{x}$. This completes the proof for the converse part.
\end{proof}
\begin{remark}\label{kg5}
By Lemma \ref{lm6}, it is easy to note that for any proper convex IVF $\textbf{F}(x)=[\underline{F}(x),\overline{F}(x)]$ and $\bar{x}\in\inte(\dom(\textbf{F}))$, $\boldsymbol \partial \textbf{F}(\bar{x})$ is nonempty.
\end{remark}
\begin{theorem}\label{17} Let $X$ be a nonempty convex subset of $\mathbb{R}^{n}$ and  $\textbf{F}:X \rightarrow {\overline{I(\mathbb{R})}}$ be a proper convex IVF with $\textbf{F}(x)=[\underline{F}(x),\overline{F}(x)]$, where $\underline{F},~ \overline{F}:X\rightarrow \overline{\mathbb{R}}$ are extended real-valued functions. Then, at any $\bar x\in \inte(\dom(\textbf{F}))$,
	\[
	\textbf{F}_\mathscr{D}(\bar{x})(h)=\boldsymbol\psi^{*}_{\boldsymbol \partial \textbf{F}(\bar{x})}(h) \text{ for all } h\in\mathbb{R}^{n} \text{ such that }\bar{x}+h\in X,
		\]
		where  $\textbf{F}_\mathscr{D}(\bar{x})(h)$ is $gH$-directional derivative of \textbf{F} at $\bar x$ in the direction of $h$.
	\end{theorem}
	
	\begin{proof}
	Note that $\underline{F}(x)~\text{and}~\overline{F}(x)$ are proper convex, and therefore  $\partial\underline{F}(\bar x)$ and $\partial\overline{F}(\bar x)$ are nonempty. Let $\underline{g}=(\underline{g}_{1},\underline{g}_{2},\ldots,\underline{g}_{n})\in\partial\underline{F}(\bar x)$ and $\overline{g}=(\overline{g}_{1},\overline{g}_{2},\ldots,\overline{g}_{n})\in\partial\overline{F}(\bar{x})$ for $\bar{x}\in\inte(\dom(\textbf{F}))$. By the property of real-valued convex functions (see \cite{dhara2011optimality}), we have
	\[
	\underline{F}_{\mathscr{D}}(\bar x)(h)=\psi^{*}_{\partial\underline{F}(\bar{x})}(h) \text{ and } \overline{F}_{\mathscr{D}}(\bar{x})(h)=\psi^{*}_{\partial\overline{F}(\bar{x})}(h) \text{ for all } h\in \mathbb{R}^{n} \text{ such that } \bar{x}+h\in X.
	\]
	Due to Theorem \ref{thm4}, we get
	\begin{align}\label{eq56}
		\textbf{F}_\mathscr{D}(\bar{x})(h)=&\Big[\min\Big\{\underline{F}_{\mathscr{D}}(\bar{x})(h),\overline{F}_{\mathscr{D}}(\bar{x})(h)\Big\}, \max\Big\{\underline{F}_{\mathscr{D}}(\bar{x})(h),\overline{F}_{\mathscr{D}}(\bar{x})(h)\Big\}\Big]\nonumber\\=&\Big[\min\Big\{\psi^{*}_{\partial\underline{F}(\bar{x})}(h),\psi^{*}_{\partial\overline{F}(\bar{x})}(h)\Big\}, \max\Big\{\psi^{*}_{\partial\underline{F}(\bar{x})}(h),\psi^{*}_{\partial\overline{F}(\bar{x})}(h)\Big\}\Big].
	\end{align}
	We now consider the following two possible cases.
	\begin{enumerate}[$\bullet$ \text{Case} 1.]
	    \item\label{kg6}  Let $\psi^{*}_{\partial\underline{F}(\bar{x})}(h)\leq\psi^{*}_{\partial\overline{F}(\bar{x})}(h)$. 
	    In this case, by (\ref{eq56}), we get 
	    \begin{align}\label{eq57}
	       	\textbf{F}_\mathscr{D}(\bar{x})(h) & = \Big[\psi^{*}_{\partial\underline{F}(\bar{x})}(h),\psi^{*}_{\partial\overline{F}(\bar{x})}(h)\Big] 
	       	= \left[\sup_{\underline{g}\in\partial\underline{F}(\bar{x})} h^{\top}\odot\underline{g}, \sup_{\overline{g}\in\partial\overline{F}(\bar{x})} h^{\top}\odot \overline{g}\right] \nonumber\\ & = 
	       \sup_{\underline{g}\in\partial\underline{F},~\overline{g}\in\partial\overline{F}(\bar{x})}\Big[h^{\top}\odot\underline{g}, h^{\top}\odot \overline{g}\Big] = 	         \sup_{\underline{g}\in\partial\underline{F},~\overline{g}\in\partial\overline{F}(\bar{x})}\left[\sum\limits_{i=1}^{n} h_{i}\underline{g}_{i}, \sum\limits_{i=1}^{n} h_{i}\overline{g}_{i}\right].
	         \end{align}
 We have seen in Lemma \ref{lm6} that corresponding to  every $\underline{g}\in\partial\underline{F}(\bar{x})$ and   $\overline{g}\in\partial\overline{F}(\bar{x})$, we get $\widehat{\textbf{G}}\in\boldsymbol \partial \textbf{F}(\bar{x})$ and vice-versa. Thus, for  $\underline{g}=(\underline{g}_{1},\underline{g}_{2},\ldots,\underline{g}_{n})\in\partial\underline{F}(\bar x)$ and $\overline{g}=(\overline{g}_{1},\overline{g}_{2},\ldots,\overline{g}_{n})\in\partial\overline{F}(\bar{x})$,  by  (\ref{eq57}), we obtain
\begin{align}
    	\textbf{F}_\mathscr{D}(\bar{x})(h)=&\sup_{\widehat{\textbf{G}}\in\boldsymbol \partial \textbf{F}(\bar{x})} h^{\top}\odot\widehat{\textbf{G}}, \text{ where } \widehat{\textbf{G}}=(\textbf{G}_{1},\textbf{G}_{2},\ldots,\textbf{G}_{n})\text{ with } \textbf{G}_{i}=[\underline{g}_{i},\overline{g}_{i}]\nonumber\\=&\boldsymbol\psi^{*}_{\boldsymbol \partial \textbf{F}(\bar{x})}(h)\nonumber.
    	\end{align}
  \item\label{kg7}  Let $\psi^{*}_{\partial\overline{F}(\bar{x})}(h)\leq\psi^{*}_{\partial\underline{F}(\bar{x})}(h)$.
   Proof contains similar steps as in Case \ref{kg6}.
  \end{enumerate}
  Hence, from Case \ref{kg6}  and Case \ref{kg7}, we have 
  	\[
	\textbf{F}_\mathscr{D}(\bar{x})(h)=\boldsymbol\psi^{*}_{\boldsymbol \partial \textbf{F}(\bar{x})}(h) \text{ for all } h\in\mathbb{R}^{n} \text{ such that }\bar{x}+h\in X.
		\]
    \end{proof}

\begin{theorem}\label{thm6}
    Let $\textbf{F}: X\rightarrow \overline{I({\mathbb{R}})}$ be a proper convex IVF and $\bar{x}\in\inte(\dom(\textbf{F}))$. Then, the $gH$-subdifferential set of $\textbf{F}$ at $\bar{x}$ is bounded.
\end{theorem}
\begin{proof}
Note that  by Theorem \ref{thm4}, for $\bar{x}\in\inte(\dom(\textbf{F}))$, the directional derivative of $\textbf{F}$ at $\bar{x}$ exists everywhere. Thus, for all $h\in\mathbb{R}^{n}$ such that $\bar{x}+h\in X$, we have
\begin{eqnarray*}
&&\textbf{F}_\mathscr{D}(\bar{x})(h) \text{ is finite }\\
&\implies& \boldsymbol\psi^{*}_{\boldsymbol\partial\textbf{F}(\bar{x})}(h) ~\text{is finite} \text{ by Theorem } \ref{17}\\
&\implies& \boldsymbol\partial\textbf{F}(\bar{x}) \text{ is bounded by Lemma }\ref{lm11}.
\end{eqnarray*}
Hence, the $gH$-subdifferential set of $\textbf{F}$ at $\bar{x}\in\inte(\dom(\textbf{F}))$ is bounded, i.e., for every $\widehat{\textbf{G}}\in\boldsymbol\partial\textbf{F}(\bar x)$, there exists an $M>0$ such that $\lVert \widehat {\textbf{G}} \rVert_{I(\mathbb{R})^{n}}\leq M$.
\end{proof}
   \begin{lemma}\label{lm10}
        Let \textbf{F} be an IVF on a nonempty set $X\subseteq\mathbb{R}^{n}$ such that 
        \[
        \textbf{F}(x)\ominus_{gH} \textbf{F}(y)\preceq {c}\lVert x-y \rVert \text{ for all } x,y\in X,
        \]
        where $c\in\mathbb{R}$. Then, 
        \[
        \lVert \textbf{F}(x)\ominus_{gH}\textbf{F}(y) \rVert_{I(\mathbb{R})}\leq c\lVert x-y\rVert \text{ for all } x,y\in X.
        \]
\end{lemma}
\begin{proof}
We have $\textbf{F}(x)\ominus_{gH} \textbf{F}(y)\preceq c\lVert x-y \rVert \text{ for all } x,y\in X,$ which implies that
\begin{eqnarray}\label{eq31}
&&\underline{F}(x)-\underline{F}(y)\leq c\lVert x-y\rVert \text{ and } \overline{F}(x)-\overline{F}(y)\leq c\lVert x-y\rVert. 
\end{eqnarray}
Interchanging $x$ and $y$ in (\ref{eq31}), we obtain
\begin{eqnarray}\label{eq32}
&&\underline{F}(y)-\underline{F}(x)\leq c\lVert x-y\rVert \text{ and } \overline{F}(y)-\overline{F}(x)\leq c\lVert x-y\rVert. 
\end{eqnarray}
With the help of  (\ref{eq31}) and (\ref{eq32}), we get
\[
\lvert \underline{F}(x)-\underline{F}(y)\rvert \leq c\lVert x-y\rVert \text{ and } \lvert \overline{F}(x)-\overline{F}(y)\rvert \leq c\lVert x-y\rVert \text{ for all } x,y\in X,
\]
which implies
\[
        \lVert \textbf{F}(x)\ominus_{gH}\textbf{F}(y) \rVert_{I(\mathbb{R})}\leq c\lVert x-y\rVert \text{ for all } x,y\in X.
        \]

\end{proof}
\begin{theorem}
    Let $X$ be a nonempty convex subset of $\mathbb{R}^{n}$ and $\textbf{F}$ be a convex IVF on $X$  such that $\textbf{F}$ has $gH$-subgradient at every ${x}\in X$. Then, $\textbf{F}$ is $gH$-Lipschitz continuous on $X$. 
\end{theorem}
\begin{proof}
Since $\textbf{F}$ has $gH$-subdradient at every $x\in X$, then there exists a $\widehat{\textbf{G}}\in I(\mathbb{R})^{n}$ such that 
\begin{eqnarray*}
&&(y-{x})^{\top}\odot\widehat{\textbf{G}}\preceq\textbf{F}(y)\ominus_{gH}\textbf{F}({x})\text{ for all } y\in X\\
&\implies& (-1)\odot\left((x-y)^{\top}\odot\widehat{\textbf{G}}\right)\preceq \textbf{F}(y)\ominus_{gH} \textbf{F}(x)\\
&\implies& \textbf{F}(x) \ominus_{gH} \textbf{F}(y) \preceq (x-y)^{\top}\odot\widehat{\textbf{G}}\\
&\implies& \textbf{F}(x) \ominus_{gH} \textbf{F}(y) \preceq\lVert x-y \rVert\odot\left[\lVert \widehat {\textbf{G}} \rVert_{I(\mathbb{R})^{n}}, \lVert \widehat {\textbf{G}} \rVert_{I(\mathbb{R})^{n}} \right] \text{ by Lemma } \ref{lm9}\\
&\implies&  \lVert \textbf{F}(x) \ominus_{gH} \textbf{F}(y)\rVert_{I(\mathbb{R})} \leq \lVert \widehat {\textbf{G}} \rVert_{I(\mathbb{R})^{n}}\lVert x-y \rVert \text{ by  Lemma } \ref{lm10}\\
&\implies& \lVert \textbf{F}(x) \ominus_{gH} \textbf{F}(y)\rVert_{I(\mathbb{R})} \leq M \lVert x-y \rVert, \text{ where } \lVert \widehat {\textbf{G}} \rVert_{I(\mathbb{R})^{n}}\leq M \text{ by Theorem \ref{thm6}}.
\end{eqnarray*}
Thus, $\textbf{F}$ is $gH$-Lipschitz continuous on X.
\end{proof}

\section{Weak sharp minima and its characterizations}\label{Sec4}
	In this section, we present the main results\textemdash primal and dual characterizations of WSM for a $gH$-lsc and convex IVF.



\begin{definition}{(\emph{WSM for an IVF})}\label{k4}. Let $\textbf{F}$ : ${\mathbb{R}}^n \rightarrow {\overline{I(\mathbb{R})}}$ be a $gH$-lsc and convex IVF. Let $\bar{S}$ and $S$ be two nonempty closed convex sets such that $\bar{S}\subseteq S\subseteq {\mathbb{R}}^n$.  Further, let $\text{dom}(\textbf{F}) \cap S \neq\emptyset$.
Then, the set $\bar{S}$ is said to be a set of WSM of $\textbf{F}$ over the set $S$ with modulus $\alpha>0$ if 
\begin{equation*}\label{eq1}
    \textbf{F}(\overline{x})\oplus\alpha \dis(x,\bar{S})\preceq\textbf{F}(x)\text{ for all}~ \overline{x}\in \bar{S}\text{ and}~ x\in S.
\end{equation*}
\end{definition}	
\begin{remark}\label{nt7}
Let $\textbf{F}$ : ${\mathbb{R}}^n \rightarrow {\overline{I(\mathbb{R})}}$ be a $gH$-lsc and convex IVF with $\textbf{F}(x)=[\underline{F}(x),\overline{F}(x)]$ for all $x\in\mathbb{R}^{n}$, where $\underline{F}, \overline{F} : {\mathbb{R}}^n\rightarrow\overline{\mathbb{R}}$ be two extended real-valued functions.  Then, $\bar{S}$ is a set of WSM of $\textbf{F}$ over $S$ with modulus $\alpha>0$ if and only if  $\bar{S}$ is a set of WSM of $\underline{F}$ and $\overline{F}$ over $S$ with modulus $\alpha>0$.
The reason is as follows.
By Remark \ref{rm1} and Lemma \ref{lm2}, it is easy to see that the functions $\underline{F}$ and $\overline{F}$ are lsc and convex. Let  $\bar{S}$ be a set of WSM of \textbf{F} over $S$ with modulus $\alpha>0$. Then, 
\begin{eqnarray*}
&&\textbf{F}(\bar{x})\oplus\alpha\dis(x,\bar{S})\preceq\textbf{F}(x)~\text{for all}~ \bar{x}\in \bar{S}\text{ and}~ x\in S\\&\iff&
[\underline{F}(\bar{x})+\alpha\dis(x,\bar{S}),\overline{F}(\bar{x})+\alpha\dis(x,\bar{S})]\preceq [\underline{F}(x),\overline{F}(x)]~\text{for all}~ \bar{x}\in \bar{S}\text{ and}~ x\in S\\&\iff&
\underline{F}(\bar{x})+\alpha\dis(x,\bar{S})\leq \underline{F}(x) \text{ and } \overline{F}(\bar{x})+\alpha\dis(x,\bar{S})\leq\overline{F}(x)~\text{for all}~ \bar{x}\in \bar{S}~\text{and}~ x\in S\\&\iff&
\bar{S} \text{ is a set of WSM of both }
\underline{F} \text{ and } \overline{F}    \text{ over } S  \text{ with modulus }\alpha>0.  
\end{eqnarray*} 
\end{remark}
\begin{example}
Let $\textbf{F}:\mathbb{R}^2\rightarrow {\overline{I(\mathbb{R})}}$ be an IVF defined by 
$$\textbf{F}(x)=[{5-x_{1}x_{2}-x_{1}},  {10-x^{2}_{1}}x_{2}-x^{2}_{2}x_{1}].$$
Let $S=[-a,0]\times[-a,0]\subseteq\mathbb{R}^2$ and $\bar{S}=\{0\}\times[-a,0]$, where $a>0$. Thus,  $\bar{S}\subseteq S\subseteq\mathbb{R}^{n}$. 
Clearly, the functions $\underline{F}$ and $\overline{F}$ are  ${5-x_{1}x_{2}-x_{1}}$ and ${10-x^{2}_{1}x_{2}-x^{2}_{2}x_{1}}$, respectively. Note that for any $\alpha>0$,
\[
\underline{F}(\bar{x})+\alpha\dis(x,\bar{S})\leq \underline{F}(x) \text{ and } \overline{F}(\bar{x})+\alpha\dis(x,\bar{S})\leq\overline{F}(x)~\text{for all}~ \bar{x}\in \bar{S}~\text{and}~ x\in S.
\]
Thus, $\bar{S}=\{0\}\times[-a,0]$ is a set of WSM of both $\underline{F}$ and $\overline{F}$ over $S$ with modulus $\alpha$, for any $\alpha>0$. Therefore, by Remark \ref{nt7},  $\bar{S}$ is a set of WSM  of $\textbf{F}$ over $S$ with modulus $\alpha>0$.
\end{example}
		    \begin{theorem}\emph{(Primal characterization)}.\label{k3}
		    Let $\textbf{F},~S$, and $\bar{S}$ be as in Definition $\ref{k4}$. Further, define an IVF $\textbf{F}_{o}:{\mathbb{R}}^n\rightarrow {\overline{I(\mathbb{R})}}$ by 
		   \begin{equation*}
		       \textbf{F}_{o}=\begin{cases}
		\textbf{F}(x), & \text{if } x \in S, \\
		+\infty, & \text{otherwise}.
		\end{cases}
		   \end{equation*} 
		   Then, the set $\bar{S}$ is a set of \textit{WSM} of $\textbf{F}$ over the set $S$ with modulus $\alpha>0$ if and only if 
		    \begin{equation}\label{eq8}
		        \alpha \dis(d,T_{\bar{S}}(x))\preceq\textbf{F}_{o\mathscr{D}}(x)(d)~\text{for all } x\in\bar{S}~ \text{and}~ d\in{\mathbb{R}}^n.
		    \end{equation}
		    \end{theorem}
		    \begin{proof} 
		    Suppose $\bar{S}$ is a set of WSM of $\textbf{F}$ over $S$ with modulus $\alpha>0$. Then, by Definition \ref{k4}, for any $x\in\bar{S},~d\in {\mathbb{R}}^n,$ and $t>0$, we have
		    \begin{eqnarray*}
		    &&\textbf{F}_{o}(x)\oplus\alpha \dis(x+td,\bar{S})\preceq\textbf{F}_{o}(x+td)\\&\implies&\alpha \dis(x+td,\bar{S})\preceq
		    \textbf{F}_{o}(x+td)\ominus_{gH}\textbf{F}_{o}(x)\\&\implies&\frac{\alpha}{t}\left( {\dis(x+td,\bar{S})-\dis(x,\bar{S})}\right)\preceq\frac{1}{t}\odot\left({\textbf{F}_{o}(x+td)\ominus_{gH}\textbf{F}_{o}(x)}\right)\\&\implies&\lim_{t\rightarrow 0}\frac{\alpha}{t}\left( {\dis(x+td,\bar{S})-\dis(x,\bar{S})}\right)\preceq\lim_{t\rightarrow 0}\frac{1}{t}\odot\left({\textbf{F}_{o}(x+td)\ominus_{gH}\textbf{F}_{o}(x)}\right)\\&\implies&\alpha\lim_{t\rightarrow 0}\frac{1}{t}\left( {\dis(x+td,\bar{S})-\dis(x,\bar{S})}\right)\preceq\textbf{F}_{o\mathscr{D}}(x)(d)~\text{by Definition \ref{d1}}\\&\implies&\alpha \dis(d,T_{\bar{S}}(x))\preceq\textbf{F}_{o\mathscr{D}}(x)(d)~\text{by part (\ref{dst4}) of Lemma \ref{aa}}.
		    \end{eqnarray*}
Thus,
		    \[ \alpha \dis(d,T_{\bar{S}}(x))\preceq\textbf{F}_{o\mathscr{D}}(x)(d) ~\text{for all } x\in\bar{S}~ \text{and}~ d\in{\mathbb{R}}^n.\]
 For the converse part, let $y\in S$ and $x\in \bar{S}$. Therefore, from Lemma \ref{lm3}, we get
 \begin{eqnarray}
 &&\textbf{F}_{o\mathscr{D}}({x})(y-x)\preceq\textbf{F}_{o}(y)\ominus_{gH}\textbf{F}_{o}(x)\nonumber\\&\implies&\textbf{F}_{o}(x)\oplus\textbf{F}_{o\mathscr{D}}({x})(y-x)\preceq \textbf{F}_{o}(y)\nonumber\\&\implies&\textbf{F}_{o}(x)\oplus\alpha \dis(y-x,T_{\bar{S}}(x))\preceq\textbf{F}_{o}(y)\text{ by (\ref{eq8})}\nonumber\\&\implies&\textbf{F}_{o}(x)\oplus\alpha \dis(y,x+T_{\bar{S}}(x))\preceq\textbf{F}_{o}(y)\nonumber\label{eq24}.
 \end{eqnarray}
Since $x\in\bar{S}$ is arbitrary, we have 
\begin{eqnarray*}
   && \textbf{F}_{o}(x)\oplus\alpha\sup_{x\in \bar{S}} \dis(y,x+T_{\bar{S}}(x))\preceq\textbf{F}_{o}(y)\\&\implies&\textbf{F}_{o}(x)\oplus\alpha \dis(y,\bar{S})\preceq\textbf{F}_{o}(y)~\text{for all}~ x\in\bar{S}~ \text{and}~\alpha>0~\text{by part (\ref{dst3})\text{ of Lemma} \ref{aa}}.
\end{eqnarray*}
Hence, $\bar{S}$ is the set of WSM of $\textbf{F}$ over $S$ with modulus $\alpha>0$, and the proof is complete.
\end{proof}
 \begin{theorem}\emph{(Dual characterizations)}.\label{th2}
  Let $\textbf{F},~S$, and $\bar{S}$ be as in Definition $\ref{k4}$. 
 Define an IVF $\textbf{F}_{o}:{\mathbb{R}}^n\rightarrow \overline{I(\mathbb{{R}})}$ by 
		   \begin{equation*}
		       \textbf{F}_{o}(x)=\begin{cases}
		\textbf{F}(x), & \text{if } x \in S, \\
		+\infty, & \text{otherwise}.
		\end{cases}
		   \end{equation*} Then, for any $\alpha>0$, the following statements are equivalent.
 \begin{enumerate}[\normalfont(a)]
     \item\label{st1} The set $\bar{S}$ is a set of WSM of \textbf{F} over the set $S$ with modulus $\alpha$.  
  \item\label{st2} The normal cone inclusion holds. That is,
\[
\alpha\mathbb{B}\cap N_{\bar{S}}(x)\subseteq\boldsymbol \partial \textbf{F}_{o}(x)  \text{ for all } x\in\bar{S}.
\]

  \item\label{st3} For all $x\in \bar{S}$ and $d\in T_{S}(x)$,
 \[
\alpha \dis(d,T_{\bar{S}}(x))\preceq \textbf{F}_{\mathscr{D}}(x)(d).
 \]
 \item\label{st4} The following inclusion holds, 
 \[
 \alpha\mathbb{B}\bigcap\left(\bigcup\limits_{x\in\bar{S}} N_{\bar{S}}(x)\right)\subseteq\bigcup\limits_{x\in\bar{S}}\boldsymbol \partial \textbf{F}_{o}(x).
\]
%
\item\label{st5} For all $x\in\bar{S}$ and $d\in T_{S}(x)\cap N_{\bar{S}}(x),$
\[
\alpha\lVert d\rVert\preceq\textbf{F}_{\mathscr{D}}(x)(d).
\]
\item\label{st7} For all $y\in S$, 
\[
\alpha \dis(y,\bar{S})\preceq\textbf{F}_{\mathscr{D}}(p)(y-p),
\]
where $p\in P(y~|~\bar{S})$.
\end{enumerate}
\end{theorem}
\begin{proof} 
(\ref{st1})$\iff$(\ref{st2}). Let $x\in\bar{S}$. By hypothesis, $\bar{S}$ is a set of WSM of $\textbf{F}$ over $S$. Therefore, by Theorem \ref{k3}, we get
\[
 \alpha \dis(d,T_{\bar{S}}(x))\preceq\textbf{F}_{o\mathscr{D}}(x)(d) ~\text{for all}~ d\in {\mathbb{R}}^n,
\]
which along with Theorem  \ref{17} imply 
\begin{equation}\label{eq5}
\alpha \dis(d,T_{\bar{S}}(x))\preceq\boldsymbol\psi^{*}_{\boldsymbol \partial \textbf{F}_{o}({x})}(d) ~\text{for all} ~d\in {\mathbb{R}}^n.
\end{equation}
Notice that {for all} $x\in \bar{S}$ {and} $d \in {\mathbb{R}^{n}}$, we have  
\begin{eqnarray*} 
\alpha \dis(d,T_{\bar{S}}(x))&=&\alpha\psi^{*}_{\mathbb{B}\cap N_{\bar{S}}(x)}(d)~\text{by (\ref{dst4}) of Lemma \ref{aa}}\\&=&
\alpha\sup\langle z,d\rangle, \text{ where } z \in \mathbb{B}\cap N_{\bar{S}}(x)\\&=& \sup\langle\alpha z,d\rangle, \text{ where}~  z \in \mathbb{B}\cap N_{\bar{S}}(x) ~\text{and}~\alpha>0\\&=&\sup\langle z,d\rangle, \text{ where } z \in \alpha\mathbb{B}\cap N_{\bar{S}}(x)\\&=&\psi^{*}_{\alpha\mathbb{B}\cap N_{\bar{S}}(x)}(d)~ \text{for all} ~d\in {\mathbb{R}}^n.
\end{eqnarray*}
That is, 
\begin{equation} \label{eq12}
    \alpha \dis(d,T_{\bar{S}}(x))=\psi^{*}_{\alpha\mathbb{B}\cap N_{\bar{S}}(x)}(d)~ \text{for all}  ~d\in {\mathbb{R}}^n.
\end{equation}
Thus, by (\ref{eq5}) and (\ref{eq12}), we get
\begin{equation}\label{eq7}
    \psi^{*}_{\alpha\mathbb{B}\cap N_{\bar{S}}(x)}(d)\preceq\boldsymbol\psi^{*}_{\boldsymbol \partial \textbf{F}_{o}({x})}(d).
\end{equation}
 Next, with the help of Lemma \ref{lm5}, we get the desired result
\begin{equation} \label{eq9}
    \alpha\mathbb{B}\cap N_{\bar{S}}(x)\subseteq\boldsymbol \partial \textbf{F}_{o}(x)~ \text{for all}~ d\in {\mathbb{R}}^n.
\end{equation}
Conversely, we have
\begin{eqnarray}
&&\alpha\mathbb{B}\cap N_{\bar{S}}(x)\subseteq\boldsymbol \partial \textbf{F}_{o}(x)  \text{ for all } x\in\bar{S}\\&\implies&\psi^{*}_{\alpha\mathbb{B}\cap N_{\bar{S}}(x)}(d)\preceq\boldsymbol\psi^{*}_{\boldsymbol \partial \textbf{F}_{o}({x})}(d)~\text{for all} ~d\in {\mathbb
R}^n~\text{by Lemma \ref{g40}}\nonumber\\&\implies& \alpha \dis(d,T_{\bar{S}}(x)) \preceq \boldsymbol\psi^{*}_{\boldsymbol \partial \textbf{F}_{o}({x})}(d) ~\text{for all} ~d\in {\mathbb{R}}^n ~\text{by}~ (\ref{eq12}).\label{eqq1}
\end{eqnarray}
Also, by Theorem \ref{17}, we have 
$$\boldsymbol\psi^{*}_{\boldsymbol \partial \textbf{F}_{o}({x})}(d)=\textbf{F}_{o\mathscr{D}}({x})(d)~ \text{for all}~d\in {\mathbb{R}}^n. $$
Thus, by (\ref{eqq1}), we get
$$ \alpha \dis(d,T_{\bar{S}}(x))\preceq \textbf{F}_{o\mathscr{D}}({x})(d)~ \text{for all}~d\in {\mathbb{R}}^n.$$
Therefore, by Theorem \ref{k3}, $\bar{S}$ is a set of WSM of \textbf{F} over $S$ with modulus $\alpha$.  
$\\$
(\ref{st1})$\iff$(\ref{st3}). Let the statement (\ref{st1}) holds. Let $ x\in\bar{S}$. Therefore, by Theorem \ref{k3}, we have
\[
\alpha \dis(d,T_{\bar{S}}(x))\preceq\textbf{F}_{o\mathscr{D}}({x})(d)~ \text{for all}~d\in T_{S}(x).
\]
Note that for $x\in\bar{S}$, $\textbf{F}_{o}(x) = \textbf{F}(x)$. Thus, 
\begin{equation}\label{eqq2}
    \textbf{F}_{o\mathscr{D}}(x)(d)=\textbf{F}_{\mathscr{D}}(x)(d)~ \text{for} ~x\in\bar{S}
~\text{and}~ d\in T_{S}(x).
\end{equation}
By (\ref{eqq2}) and Theorem \ref{k3}, we get
\begin{equation*}\label{eq10}
    \alpha \dis(d,T_{\bar{S}}(x))\preceq\textbf{F}_{\mathscr{D}}({x})(d)~ \text{for all}~d\in T_{S}(x) ~\text{and}~ x\in \bar{S}.
\end{equation*}
Conversely, we are given that  
\begin{eqnarray}\label{eq67}
&\alpha \dis(d,T_{\bar{S}}(x))\preceq \textbf{F}_{\mathscr{D}}(x)(d)~\text{for all}~x\in \bar{S}~ \text{and}~ d\in T_{S}(x)\nonumber\\\implies&  \alpha \dis(d,T_{\bar{S}}(x))\preceq\boldsymbol\psi^{*}_{\boldsymbol \partial \textbf{F}(x)}(d) ~\text{for all} ~d\in T_{S}(x)~ \text{by Theorem \ref{17}}.
\end{eqnarray}
Note that for $x\in\bar{S}$, we have
\begin{align}\label{eq25}
    \boldsymbol\psi^{*}_{\boldsymbol \partial \textbf{F}(x)}(d) =
\boldsymbol\psi^{*}_{\boldsymbol \partial \textbf{F}_{o}(x)}(d) ~\text{for all}~ d\in {\mathbb{R}}^n.
\end{align}
In view of  (\ref{eq67}) and (\ref{eq25}), we have 
\begin{eqnarray*}
&&\alpha\dis(d,T_{\bar{S}}(x))\preceq\boldsymbol\psi^{*}_{\boldsymbol \partial \textbf{F}_{o}(x)}(d) ~\text{for all} ~d\in T_{S}(x)\\
&\implies&\alpha\dis(d,T_{\bar{S}}(x))\preceq\textbf{F}_{o\mathscr{D}}({x})(d)~ \text{for all}~d\in T_{S}(x) \text{ by Theorem }\ref{17}.
\end{eqnarray*}
Hence, by Theorem \ref{k3}, $\bar{S}$ is the set of WSM of \textbf{F} over S with modulus $\alpha>0$.
$\\$
(\ref{st2})$\iff$(\ref{st4}). If the statement (\ref{st2}) holds, then obviously the statement (\ref{st4}) also holds.
$\\$
 Conversely, let the statement (\ref{st4}) holds. Let $x\in \bar{S}$ and
$\widehat{{G}}\in \alpha\mathbb{B}\cap N_{\bar{S}}(x)$. Therefore, there exists a $\bar{y}\in \bar{S}$ such that $\widehat{{G}}\in \boldsymbol \partial \textbf{F}_{o}(\bar{y})$. Thus, by Definition \ref{dd5}, we get
\begin{equation}\label{eq17}
(z-\bar{y})^{\top}\odot\widehat{{G}}\preceq\textbf{F}_{o}(z)\ominus_{gH}\textbf{F}_{o}(\bar{y}) \text{ for all} ~z\in \mathbb{R}^{n}.\\
\end{equation}
In particular, for any $z\in \bar{S}$,    ~$\textbf{F}_{o}(z)=\textbf{F}_{o}(\bar{y})$. Thus, (\ref{eq17}) reduces to 
\[
(z-\bar{y})^{\top}\odot\widehat{{G}} \preceq\textbf{0} \text{ for all }z\in \bar{S}.
\]
Since $\widehat{{G}}\in\mathbb{R}^{n}$, by using Remark \ref{n1},
$(z-\bar{y})^{\top}\odot\widehat{{G}}=\left\langle\widehat{{G}},z-\bar{y}\right\rangle\leq{0} \text{  for all } z\in \bar{S}$.
Therefore,
\begin{align}\label{eq18}
    &\left\langle\widehat{{G}}, z\right\rangle\leq \left\langle \widehat{{G}}, \bar{y}\right\rangle \text{  for all } z\in \bar{S}\nonumber\\
    \implies& \sup\limits_{z\in\bar{S}} \left\langle\widehat{{G}},z\right\rangle\leq\left\langle \widehat{{G}}, \bar{y}\right\rangle\nonumber\\
    \implies&\psi^{*}_{\bar{S}}(\widehat{{G}})=\left\langle\widehat{{G}},\bar{y}\right\rangle \text{ because } \bar{y}\in\bar{S}.
    \end{align}
 Since $\widehat{{G}}\in N_{\bar{S}}(x)$, by Definition \ref{d2}, we have
\begin{align}\label{eq19}
 &\left\langle\widehat{{G}},z-x\right\rangle\leq{0} \text{  for all } z\in \bar{S}\nonumber\\
 \implies& \psi^{*}_{\bar{S}}(\widehat{{G}})=\left\langle\widehat{{G}},x\right\rangle.
\end{align}
Combining (\ref{eq18}) and (\ref{eq19}), we get
\begin{align}\label{eq69}
 & \left\langle\widehat{{G}},x\right\rangle=\left\langle\widehat{{G}},\bar{y}\right\rangle.
\end{align}
Note that 
\begin{eqnarray*}
(z-x)^{\top}\odot\widehat{{G}}&=&\left\langle\widehat{{G}},z-x\right\rangle~\text{for all} ~z\in \mathbb{R}^{n}\\
&=&\left\langle\widehat{{G}},z-\bar{y}\right\rangle~\text{for all} ~z\in \mathbb{R}^{n} \text{ by } (\ref{eq69})\\
&=& (z-\bar{y})^{\top}\odot\widehat{{G}} ~\text{for all} ~z\in \mathbb{R}^{n} \text{ by Remark } \ref{n1}\\
&\preceq&\textbf{F}_{o}(z)\ominus_{gH}\textbf{F}_{o}(\bar{y}) \text{ for all} ~z\in \mathbb{R}^{n} \text{ by } (\ref{eq17})\\
&=&\textbf{F}_{o}(z)\ominus_{gH}\textbf{F}_{o}(x) \text{ for all} ~z\in \mathbb{R}^{n} \text{ becuase } \textbf{F}(x)=\textbf{F}(\bar{y}). 
\end{eqnarray*}
Hence, $\widehat{{G}}\in\boldsymbol \partial \textbf{F}_{o}(x)$. Since $x\in \bar{S}$ is arbitrary, the  statement (\ref{st2}) holds.
$\\$
(\ref{st3})$\implies$(\ref{st5}). 
From the statement (\ref{st3}), we have
\begin{eqnarray*}
&&\alpha \dis(d,T_{\bar{S}}(x))\preceq \textbf{F}_{\mathscr{D}}(x)(d) \text{ for all }d\in T_{S}(x) \text{ and } x\in \bar{S}\\
&\implies& \alpha\lVert d\rVert\preceq \textbf{F}_{\mathscr{D}}(x)(d) \text{ for all } d\in T_{S}(x)\cap N_{\bar{S}}(x) \text{ by (\ref{dst4})} \text{ of Lemma }\ref{aa}.
\end{eqnarray*}
 Hence, the statement (\ref{st5}) holds.
  $\\$
(\ref{st5})$\implies$(\ref{st1}). Let $y\in S$. Set ${x}=P(y~|~\bar{S})$, then $(y-{x})\in T_{S}(x)\cap N_{\bar{S}}(x)$. Therefore, according to the hypothesis, we obtain
\begin{eqnarray*}
&&\alpha\lVert y-x\rVert\preceq\textbf{F}_{\mathscr{D}}(x)(y-x)\\&\implies&\alpha \dis(y,\bar{S})\preceq\textbf{F}_{\mathscr{D}}(x)(y-x)\text{ by Definition } \ref{dd6}\\
&\implies& \alpha \dis(y,\bar{S})\preceq\textbf{F}(y)\ominus_{gH}\textbf{F}(x)~ \text{by Lemma } \ref{lm3}\\
&\implies& \textbf{F}(x)\oplus\alpha \dis(y,\bar{S})\preceq\textbf{F}(y) \text{ by  (\ref{lm41}) of Lemma \ref{lm4}},
\end{eqnarray*}
which shows that $\bar{S}$ is a set of WSM of \textbf{F} over $S$.
$\\$
(\ref{st1})$\iff$(\ref{st7}). Let the statement (\ref{st1}) holds. Let $y\in S$ and $p=P(y~|~\bar{S})$. Thus, the statement (\ref{st1}) gives
\begin{equation}\label{eq72}
\textbf{F}(p)\oplus\alpha \dis(y,\bar{S})\preceq\textbf{F}(y),  ~\text{i.e.,}~\textbf{F}(p)\oplus\alpha\lVert y-p\rVert\preceq\textbf{F}(y). 
\end{equation} 
Define $z_{\lambda}=\lambda y+ (1-\lambda)p$ for $\lambda\in[0,1].$ Then, $p=P(z_{\lambda}~|~\bar{S})$ for all $\lambda\in[0,1]$. From (\ref{eq72}), we have
\begin{eqnarray}\label{eq73}
&&\textbf{F}(p)\oplus\alpha\lVert z_{\lambda}-p\rVert\preceq\textbf{F}(z_{\lambda})\nonumber\\&\implies&\textbf{F}(p)\oplus\alpha\lambda\lVert y-p\rVert\preceq\textbf{F}(z_{\lambda})\nonumber\\
&\implies&\alpha\lVert x-p\rVert\preceq\frac{1}{\lambda}\odot\Big(\textbf{F}(p+\lambda(y-p))\ominus_{gH}\textbf{F}(p)\Big).
\end{eqnarray}
By taking limit as $\lambda \downarrow 0$ in (\ref{eq73}), we get 
\[
\alpha \dis(y,\bar{S})\preceq\textbf{F}_{\mathscr{D}}(p)(y-p), \text{ where } p\in P(y~|~\bar{S}).
\]
Conversely, let $y\in S$ and set ${x}=P(y~|~\bar{S})$. Then, from the statement (\ref{st7}),  we get
\begin{eqnarray*}
&&\alpha \dis(y,\bar{S})\preceq\textbf{F}_{\mathscr{D}}(x)(y-x)\\
&\implies& \alpha \dis(y,\bar{S})\preceq\textbf{F}(y)\ominus_{gH}\textbf{F}(x)\text{ by Lemma }\ref{lm3}\\
&\implies& \textbf{F}(x)\oplus\alpha \dis(y,\bar{S})\preceq\textbf{F}(y) \text{ by (\ref{lm41})} \text{ of Lemma }\ref{lm4},
\end{eqnarray*}
which is the required result.
\end{proof}
\section{Conclusion and future scopes}\label{Sec5}
In this article, the conventional concepts of support function and subdifferentiability have been extended for IVFs (Definitions \ref{dd4} and Definition \ref{dd5}). Also, some important characteristics of the $gH$-subdifferential set like nonemptyness (Lemma \ref{lm6}), boundedness (Theorem \ref{thm6}), convexity and closedness (Theorem \ref{kg9}) have been presented. Subsequently, we have provided few necessary results (Lemma \ref{g40}, Theorem \ref{lm1} and Lemma \ref{lm5}) based on the support function of a subset of $I(\mathbb{R})^{n}$. It has been reported that the $gH$-subdifferential set of a $gH$-differentiable convex IVF is a singleton set containing the $gH$-gradient (Theorem \ref{th5}). The relationship between $gH$-directional derivative and the support function of $gH$-subdifferential set of convex IVF has been also established (Theorem \ref{17}). Further, we have introduced the notion  of WSM for convex IVFs (Definition \ref{k4}). With the help of the proposed concepts of $gH$-subdifferentiability and support function, a primal characterization (Theorem \ref{k3}) and a few dual characterizations (Theorem \ref{th2}) of WSM have been presented.
\\

In future, we shall apply proposed theory on WSM to derive necessary and sufficient conditions under which a global error bound may exist for a convex inequality system as follows:
\begin{equation}\label{eq80}
    \textbf{H}_{\lambda}(x)\preceq\textbf{0}, ~\lambda\in\Lambda \text{ and } x\in C,
\end{equation}
where $\Lambda$ is an index set, and for each $\lambda\in\Lambda$, $\textbf{H}_{\lambda}: X\subseteq\mathbb{R}^{n}\rightarrow I(\mathbb{R})$ is $gH$-lsc, convex, proper, and the set $C$ is closed convex subset of $X$.
By a global error bound for the inequality system (\ref{eq80}), we mean the existence of a constant $\beta>0$ such that
\begin{equation}\label{eq81}
    \beta\dis\left(x,\Omega\right)\preceq\dis(x,C)\oplus \textbf{H}_{\lambda+}(x)  \text{ for each } \lambda\in\Lambda \text{ and } x\in X,
\end{equation}
where $\Omega =\{x:x\in C \text{ and } \textbf{H}_{\lambda}(x)\preceq\textbf{0}\}$
and $\textbf{H}_{\lambda+}(x)=\max\{\textbf{0},\textbf{H}_{\lambda}(x)\}.$ For the sake of convenience, we define $\textbf{H}(x)=\sup\left\{\textbf{H}_{\lambda}(x):\lambda\in\Lambda\right\}$ for each $x\in X$. Note that if a constant $\bar\beta>0$ exists such that 
\begin{equation}\label{eq82}
    \bar\beta\dis\left(x,\Omega\right)\preceq\dis(x,C)\oplus \textbf{H}_{+}(x) \text{ for all } x\in {X},
\end{equation}
where $\bar{\Omega} =\{x:x\in C \text{ and } \textbf{H}(x)\preceq\textbf{0}\}$
and $\textbf{H}_{+}(x)=\max\{\textbf{0},\textbf{H}(x)\} $, then (\ref{eq81}) holds for $\beta>0$.
The following observation can be useful to solve the problem. If we define an IVF $\textbf{F}:{X}\rightarrow {\overline{I(\mathbb{R})}}$ such that
\[
\textbf{F}(x)=\dis(x,C)\oplus\textbf{H}_{+}(x),
\]
then \textbf{F} has $\Omega$ as a set of WSM with modulus $\bar\beta>0$, which is equivalent to the condition in (\ref{eq82}).
\appendix
\section{Proof of Lemma \ref{lm4}}
\label{aa1}(\ref{lm41}).
	Let $\textbf{A}=[\underline{a},\overline{a}], \textbf{B}=[\underline{b},\overline{b}], \text{ and } \textbf{C}=[\underline{c},\overline{c}]$. We have 
	\begin{equation}\label{eq20}
	    r\leq\underline{a} \text{ and } r\leq\overline{a}.
	    \end{equation}
	Similarly, by $\textbf{A}\preceq\textbf{B}\ominus_{gH}\textbf{C}$, we have
	\begin{equation}\label{eq21}
	    	\underline{a}\leq\min\left\{\underline{b}-\underline{c},\overline{b}-\overline{c}\right\} \text{and }  \overline{a}\leq\max\left\{\underline{b}-\underline{c},\overline{b}-\overline{c}\right\}.
	\end{equation}
	\begin{enumerate}[$\bullet$ \text{Case} 1.]
			\item\label{apcs1} Let $\min\left\{\underline{b}-\underline{c},\overline{b}-\overline{c}\right\}=\overline{b}-\overline{c}$  and $\max\left\{\underline{b}-\underline{c},\overline{b}-\overline{c}\right\}=\underline{b}-\underline{c}$. Then, from  (\ref{eq20})  and (\ref{eq21}), we get 
\[			
\overline{c}+r\leq\overline{b} \text{ and }
\underline{c}+r\leq\underline{b}.\]
Hence,
$\textbf{C}\oplus[r,r]\preceq\textbf{B}.$
\item When $\min\left\{\underline{b}-\underline{c},\overline{b}-\overline{c}\right\}=\underline{b}-\underline{c}$  and $\max\left\{\underline{b}-\underline{c},\overline{b}-\overline{c}\right\}=\overline{b}-\overline{c}$.
Proof contains similar steps as in Case \ref{apcs1}.
\end{enumerate}

(\ref{lm43}). Let $\textbf{A}=[\underline{a},\overline{a}]$ and $\textbf{B}=[\underline{b},\overline{b}].$ Then, 
\begin{eqnarray*}
\Big\{(1-\lambda)\odot\textbf{A}\oplus\lambda\textbf{B}\Big\}\ominus_{gH}\textbf{A}&=&\Big\{(1-\lambda)\odot[\underline{a},\overline{a}]\oplus\lambda\odot[\underline{b},\overline{b}]\Big\}\ominus_{gH}[\underline{a},\overline{a}]\\
&=&\left[(1-\lambda)\underline{a}+\lambda\underline{b},(1-\lambda)\overline{a}+\lambda\overline{b}\right]\ominus_{gH}[\underline{a},\overline{a}\Big]~\text{because }\lambda\in[0,1]\\
&=&\Big[\min\Big\{\lambda\underline{b}-\lambda\underline{a},\lambda\overline{b}-\lambda\overline{a}\Big\},\max\Big\{\lambda\underline{b}-\lambda\underline{a},\lambda\overline{b}-\lambda\overline{a}\Big\}\Big]\\
&=&\lambda\odot\Big\{\textbf{A}\ominus_{gH}\textbf{B}\Big\}.
\end{eqnarray*}
		\noindent 

\section*{Funding}
Not applicable.

\section*{Author contributions}
All authors contributed to the study conception and analysis. Material preparation and analysis were
performed by Krishan Kumar, Debdas Ghosh, and Gourav Kumar. The
first draft of the manuscript was written by Krishan Kumar and all authors commented on previous
versions of the manuscript. All authors read and approved the final manuscript.

\section*{Conflict of interest}

 The authors declare that they have no known competing financial interests or personal relationships
that could have appeared to influence the work reported in this paper.
\section*{Availability of data and materials}
Not applicable.
\section*{Code availability}
Not applicable.
\bibliographystyle{spmpsci}
\bibliography{KK.bib}


%
%

\end{document}